\documentclass[oneside]{amsart}

\usepackage[T1]{fontenc}

\usepackage{amsmath,amsthm,amsfonts,amssymb,latexsym, mathrsfs}
\usepackage{tikz}
\usepackage{tikz-cd}

\makeatletter

\theoremstyle{plain}
\newtheorem{thm}{Theorem}[section]

\numberwithin{equation}{section} %% Comment out for sequentially-numbered

\numberwithin{figure}{section} %% Comment out for sequentially-numbered

\theoremstyle{plain}
\newtheorem{cor}[thm]{Corollary} %%Delete [thm] to re-start numbering

\theoremstyle{plain}
\newtheorem{claim}[thm]{Claim} %%Delete [thm] to re-start numbering

\theoremstyle{definition}
\newtheorem{defn}[thm]{Definition}

\theoremstyle{definition}
\newtheorem{rem}[thm]{Remark}

\theoremstyle{plain}
\newtheorem{lem}[thm]{Lemma} %%Delete [thm] to re-start numbering

\theoremstyle{plain}
\newtheorem{prop}[thm]{Proposition} %%Delete [thm] to re-start numbering

\theoremstyle{plain}
\newtheorem{fact}[thm]{Fact}

\theoremstyle{plain}

\theoremstyle{plain}

\newtheorem{example}[thm]{Example}

\newtheorem{ques}[thm]{Question}

\newtheorem{obse}[thm]{Observation}

\newtheorem{theo}[thm]{Theorem}

\newcommand{\Qu}{\mathbb{Q}}
\newcommand{\Na}{\mathbb{N}}	

\newcommand{\Ce}{\mathbb{C}}

\newcommand{\Ao}{\aleph_0}

\providecommand{\norm}[1]{\|#1\|}
\providecommand{\Norm}[1]{\left\|#1\right\|}
\providecommand{\parc}[1]{\left(#1\right)}

\providecommand{\parl}[1]{\left\{#1\right\}}
\providecommand{\Parl}[1]{\lbrace#1\rbrace}

\providecommand{\gen}[1]{\langle#1\rangle}

\newcommand{\Lim}[2]{\lim\limits_{{#1}\to {#2}}}
\newcommand{\Sum}[2]{\displaystyle\sum_{#1}^{#2}}

\newcommand\redsout{\bgroup\markoverwith{\textcolor{red}{\rule[0.5ex]{2pt}{1.5pt}}}\ULon}%Tachar co rojo

\newcommand{\forkindep}[1][]{%
  \mathrel{
    \mathop{
      \vcenter{
        \hbox{\oalign{\noalign{\kern-.3ex}\hfil$\vert^*$\cr
              \noalign{\kern-.7ex}
              $\smile$\cr\noalign{\kern-.3ex}}}
      }
    }\displaylimits_{#1}
  }
}

\newcommand{\quotient}[2]{{\left.\raisebox{.1em}{$#1$}/\raisebox{-.1em}{$#2$}\right.}}

\ifx\pdfoutput\@undefined\usepackage[usenames,dvips]{color}

\IfFileExists{pdfcolmk.sty}{\usepackage{pdfcolmk}}{}

\fi

\makeatletter

\usepackage{amssymb}

\usepackage{amsmath}

\def \<{\langle}
\def \>{\rangle}

\usepackage{amsfonts}

\usepackage{amscd}

\usepackage{latexsym}

\usepackage{epsfig}

\usepackage{graphicx}

\makeatother

\def\Ind#1#2{#1\setbox0=\hbox{$#1x$}\kern\wd0\hbox to 0pt{\hss$#1\mid$\hss}
\lower.9\ht0\hbox to 0pt{\hss$#1\smile$\hss}\kern\wd0}

\def\Notind#1#2{#1\setbox0=\hbox{$#1x$}\kern\wd0\hbox to 0pt{\mathchardef
\nn=12854\hss$#1\nn$\kern1.4\wd0\hss}\hbox to
0pt{\hss$#1\mid$\hss}\lower.9\ht0 \hbox to
0pt{\hss$#1\smile$\hss}\kern\wd0}

\newcommand{\acl}{\operatorname{acl}}

\newcommand{\tp}{\operatorname{tp}}

\newcommand{\qftp}{\operatorname{qftp}}

\newcommand{\dcl}{\operatorname{dcl}}

\newcommand{\cl}{\operatorname{cl}}

\begin{document}

\title{Model theory of Hilbert spaces with a discrete group action}

 \date{\today}

\author{Alexander Berenstein}
\address{Universidad de los Andes,
Cra 1 No 18A-12, Bogot\'{a}, Colombia} \urladdr{www.matematicas.uniandes.edu.co/\textasciitilde aberenst}

\author{Juan Manuel P\'erez}
\address{Universidad de los Andes,
Cra 1 No 18A-12, Bogot\'{a}, Colombia}
\email{jm.perezo@uniandes.edu.co}

% \keywords{NIP theories, unary predicate expansions, Mordell-Lang property...}
% \subjclass[2010]{03C45, 03C64}
% \thanks{The authors would like to thank Jonathan Kirby and Andr\'as Villaveces for valuable feedback.}

\thanks{\hspace{-0.4cm}2020 Mathematics Subject Classification. 03C66, 03C45, 47D03 , 47C15.
\\ The authors would like to thank Tom\'as Ibarluc\'ia and Xavier Caicedo for valuable feedback as well as the referee for helping us improve the presentation of the results.}

\begin{abstract}
In this paper we study expansions of infinite dimensional Hilbert spaces with a unitary representation of a discrete countable group. When the group is finite, we prove the theory of the corresponding expansion, regardless if it is existentially closed, has quantifier elimination, is $\aleph_0$-categorical, $\aleph_0$-stable and SFB. On the other hand, when the group involved is countably infinite, the theory of the Hilbert space expanded by the representation of this group is $\aleph_0$-categorical up to perturbations. Additionally, when the expansion is model complete, we prove that it is $\aleph_0$-stable up to perturbations.

%But, when the expansion is model complete, we prove it is $\aleph_0$-categorical up to perturbations and $\aleph_0$-stable up to perturbations.
 
\end{abstract}

\pagestyle{plain}

\maketitle

\section{Introduction}

In this paper we work on model theoretic aspects of the expansion of a Hilbert space by a unitary representation of a countable discrete group. A \textit{unitary representation} of a group $G$ in a Hilbert space $H$ is defined as an action of $G$ on $H$ by elements of the group of unitary maps, denoted by $U(H)$. In other words, a representation is given by a homomorphism $\pi : G \longrightarrow U(H)$, where the action of $g \in G$ on $v \in H$ is denoted by $\pi(g)v$. 
%Given a family $\parl{\pi_i : i \in I}$ of unitary representations of $G$ acting on $\parl{H_i : i \in I}$, we define the direct sum $\tau := \oplus_i \pi_i$ as follows: let $H = \oplus_{i\in I} H_i$ and for $g \in G$ and $v \in H$, the action of $\tau(g)$ on $v$ is given by $\tau(g)(\oplus_i v_i) := \oplus_i \pi_i(g)v_i$, where $v_i$ is the projection of $v$ on $H_i$. Notice that, the pair $(H, \tau)$ is again a unitary representation. Additionally, for a representation $\pi$, the sum of $n$ copies of $\pi$ is denoted as $n\pi$, and the sum of enumerable copies of $\pi$ is denoted as $\infty \pi$. Therefore, the objective of the paper is to study the prinicipal model theoretic properties of the expansion of $H$ by $\pi$ or $\infty\pi$, depending on if $G$ is finite or infinite. Indeed, when the group is finite the results will be stronger than when the group is infinite, as we will see.

We treat Hilbert spaces as continuous structures in the language $\mathcal{L} = \{0,-,\dot2, \frac{x+y}{2},\\ e^{i\theta}: \theta \in 2\pi\Qu\}$, which allows to axiomatize Hilbert spaces as a universal theory (the proof is a small modification of the argument in \cite{itai2}, the proof in \cite{itai2} deals with real Hilbert spaces where one omits the family $\{e^{i\theta}: \theta \in 2\pi\Qu\}$ and for the complex case, one usually only includes $i=e^{i\pi/2}$ ) and in this language the theory has quantifier elimination. To deal with expansions by a group of unitary maps, we add a unary function symbol for each element $g$ of the group and we interpret it as $\pi(g)$. 

There are several papers that deal with similar expansions. For instance, expansions of a Hilbert space with a single automorphism were studied in \cite{zadka}, showing that the existentially closed models correspond to expansions by a unitary map with spectrum $S^1$. Moreover, it is also proved that the expansion is superstable but not $\Ao$-stable. In \cite{AlexHGGA} it is proved that if $G$ is \textit{amenable} and countable, then a Hilbert space $H$ expanded by a countable number of copies of the left regular representation of $G$ is existentially closed. It is also proved that this class of expansions has a model companion which is existentially axiomatizable. Furthermore, when $G$ is countable, this model companion is superstable. The theory of a Hilbert space expanded by a single unitary operator with countable spectrum is treated in \cite{Argoty}, where it is proved that the expansion eliminates quantifiers and is $\Ao$-stable.
Both papers \cite{zadka,Argoty} relied heavily on tools from spectral theory like the spectral decomposition theorem. 

In this paper, when the group involved is non abelian, we cannot use the spectral decomposition theorem to describe the action. Instead, we rely on tools from representation theory when dealing with finite groups or tools from $C^*$-algebras theory when the group is not finite. We denote the theory of a Hilbert space expanded by the unitary representation $\pi$ of a group $G$ as $IHS_\pi$.

We first consider expansions corresponding to actions of finite groups. The main tool we use are basic ideas from representation theory (see \cite{pierre}, which we review in Section \ref{sec:background} below). In this context, existentially closed expansions (that played a crucial role in the literature of similar expansions) can be understood as those with a richer presence of irreducible representations, other representations will have instead irreducible pieces with finite multiplicity.  We prove that the theory of any such expansion is $\Ao$-categorical and $\Ao$-stable. We also define a natural notion of \textit{independence} and prove it coincides with non-forking, which allows us to prove some ``geometric'' results associated to the theory of the expansion. For example, we show the expansion is \textit{non-multidimensional}. We also show that the associated theory of \textit{Belles Paires} of $IHS_\pi$ is $\Ao$-categorical and thus the theory $IHS_\pi$ is strongly finitely based (SFB)(see Fact \ref{fact:SBFEquivalence} and the discussion before for more details). 

Then we deal with the case when $G$ is countable infinite. To analyze these expansions, we need to switch to new tools. Instead of using tools from representation theory, we need to consider the $C^*$-algebra generated by the unitary maps from the representation and use consequences of Voiculescu's theorem (see Theorem II.5.8 \cite{ken} and Section \ref{sec:background} below) to prove that the theory $IHS_\pi$ is $\Ao$-categorical up to perturbations, and when the theory $IHS_\pi$ is model complete, it is $\Ao$-stable up to perturbations (see Definition \ref{defn:pertubations} below).

%when the theory $IHS_\pi$ is model complete, then  $IHS_\pi$ is $\Ao$-categorical up to perturbations and $\Ao$-stable up to perturbations (see \cite{BY-perturbations} for the definitions).

This paper is organized as follows. In Section \ref{sec:background} we give some basic tools from representation theory of finite groups, and basic background on operator theory including the notion of spectrum of a unitary operator, some ideas from $C^*$-algebras like Voicolescu’s Theorem, and  some model-theoretic applications to perturbations. In Section \ref{rep:finite} we consider the case where $G$ is finite, we prove the corresponding expansions $\Ao$-categorical and $\Ao$-stable and give a natural characterization of non-forking independence and show the theory is SFB. In Section \ref{rep:infinite}, we deal with the case where $G$ is infinite and prove that the theory $IHS_\pi$ is $\Ao$-categorical up to perturbations. Finally, we show that when the theory $IHS_\pi$ is model complete, then  $IHS_\pi$ is $\Ao$-stable up to perturbations. 

We will assume the reader is familiar with continuous logic, all background needed can be found in \cite{AlexMT, ItaiStability}, some basic knowledge of perturbations will also be helpful, the corresponding background can be found in \cite{{BY-perturbations}}. We will assume no prior knowledge of representation theory. The necessary background on this subject and on operator theory will be introduced in Section \ref{sec:background}.

%%%%%%%%%%%%%%%%%%%%%%%%%%%%%%%%%%%%%%%%%%%%%%%%%%%%%%%%%%%%%%%%%%%%%%%%%%%%%%%%%%%%%%%%%%%%%%%%%%%%%%%%%%%%%%%%%%%%

\section{Background on operator theory and representation theory}\label{sec:background}

In this section, we first review results from representations of finite groups, our main focus is on irreducible representations and projections onto sums of isomorphic irreducible representations; these sums play the role of basic blocks that will help us describe the theory $IHS_\pi$. We then introduce some technical tools from operator theory and $C^*$-algebras concerning Voiculescu's theorem. 

We start by introducing the language used to treat a Hilbert space expanded by a unitary group representation as a metric structure.

\begin{defn}\label{theoryoftheexpansion}

Let $G$ be a group and let $H$ be a Hilbert space. A \textit{unitary representation} $\pi$ of $G$ on $H$ is a homomorphism $\pi \colon G \longrightarrow U(H)$. We define  
\[ \mathcal{L}=\Parl{0,-,\dot2, \frac{x+y}{2},e^{i\theta}: \theta \in 2\pi\Qu }
\]

to be the language of Hilbert spaces and

\[
\mathcal{L}_\pi := \mathcal{L}\cup\parl{\pi(g) : g \in G},
\]

as the representation language, where each $\pi(g)$ is a unary function with modulus of uniform continuity $\Delta(\epsilon)=\epsilon$. We denote the theory of infinite dimensional Hilbert spaces by $IHS$. The language $\mathcal{L}$ defined above is the one presented in \cite{itai2} enriched by multiplication by the family of complex scalars $\{e^{i\theta}: \theta \in 2\pi\Qu\}$. We denote by $IHS_\pi := Th(H, \pi)$ the theory of the infinite-dimensional Hilbert space $H$ expanded by the unitary representation $\pi$ of $G$ in the language $\mathcal{L}_\pi$. Note that the theory $IHS_\pi$ includes information like "each $\pi(g)$ is a unitary map" and "for all $g_1,g_2\in G$, $\pi(g_1\cdot g_2)=\pi(g_1)\pi(g_2)$ as functions".

\end{defn}

\subsection{Representation theory of finite groups on linear groups}\label{Sec:finitegroups}

In this subsection we recall some results about representations of finite groups from \cite{pierre}. These results will be useful to prove that the theory $IHS_\pi$ is $\Ao$-categorical and $\Ao$-stable, where $\pi$ is any unitary representation of a finite group $G$ on an infinite dimensional Hilbert space. In this subsection, $G$ will always stand for a finite group and $V$ for a finite dimensional vector space.

\begin{defn}[Definition 1.1 \cite{pierre}]

A \textit{linear representation} of $G$ in $V$ is a homomorphism $\pi$ from $G$ into GL$(V)$. When $V$ has dimension $n$, the representation is said to have \textit{degree} $n$.

\end{defn}

Now, we introduce the \textit{left regular representation} of $G$. As we will later see, this representation is especially rich respect to other representations.

\begin{defn} (Example 1.2.b of \cite{pierre})
Suppose that $V$ has dimension $|G|$ with basis $\parl{e_g}_{g \in G}$ indexed by the elements of $G$ (if we add to $V$ a Hilbert space structure, it is denoted by $\ell_2(G)$). For all $h \in G$, we denote by $\lambda_G(h)$ the linear map sending each $e_g$ to $e_{hg}$; this defines a linear representation of $G$, which is called the \textit{left regular representation} of $G$ and it is denoted by $\lambda_G$.
\end{defn}

\begin{defn}[Section 1.3 \cite{pierre}]

Let $\pi : G \longrightarrow$ GL$(V)$ be a linear representation and let $W$ be a vector subspace of $V$. Assume that $W$ is \textit{invariant} under the action of $G$. Then, the restriction maps $\parl{\pi(g)\!\!\upharpoonright\!\!_W}_{g\in G}$ are automorphisms of $W$ satisfying for all $g_1, g_2 \in G$
\[
\pi(g_1g_2)\!\!\upharpoonright\!\!_W = \pi(g_1)\!\!\upharpoonright\!\!_W \pi(g_2)\!\!\upharpoonright\!\!_W.
\]
Thus, $\pi\!\!\upharpoonright\!\!_W : G \longrightarrow$ GL$(W)$ is a linear representation of $G$ in $W$ and it is called a \textit{subrepresentation} of $V$. Additionally, if $W$ has no non-proper and non-trivial subrepresentation, it is called an \textit{irreducible representation}.

\end{defn}

\begin{fact}[Theorem 2 of \cite{pierre}]\label{irreducibleDecompositionRepresentation}

Every linear representation is a direct sum of irreducible representations.

\end{fact}

\begin{fact}[Corollary 1, Theorem 4 of \cite{pierre}]\label{irreducibleDecomp-multiplicity}

The number of irreducible representations $W_i$ isomorphic to a given irreducible representation $W$ is independent of the chosen decomposition.

\end{fact}

These last two facts will allow us to understand the theory of a Hilbert space expanded by a unitary representation of a group. An arbitrary unitary representation can be split into a direct sum of its irreducible subrepresentations and we can count the number of times each irreducible representation appears in the sum. We call this number (which can be a natural number or $\infty$ when dealing with infinite dimensional Hilbert spaces) the \emph{multiplicity} of the given irreducible representation. Part of our work in section  \ref{rep:finite} is to show that the multiplicity of each irreducible representation can be recovered from the theory of the expansion. 

\begin{rem}\label{rem:betterthinkfinitely}

Let $H$ be a Hilbert space of infinite dimension and let $\pi$ be a homomorphism from $G$ to $U(H)$. Studying such a representation can always be reduced to the study of representations of $G$ in finite dimensional subspaces. To do this reduction, take some $x \in H$ and consider the finite dimensional subspace generated by $\parl{\pi(g)x}_{g \in G}$. In this space we could use the theory for representations of finite dimension and then we wrap these spaces together to understand the action of $G$ all over $H$. 

\end{rem}

\begin{fact}[Corollary 1, Proposition 5 of \cite{pierre}]\label{decompositionRegularRepresentation}

Every irreducible unitary representation $W$ of $G$ is contained in the left regular representation of $G$ with multiplicity equal to its degree. In particular, there are only finitely many irreducible unitary representations of $G$.

\end{fact}

This fact suggests that if we take the direct sum of countably many left regular representations of $G$, then the structure obtained should be existentially closed, which indeed is the case even in the larger setting of amenable groups:

\begin{fact}[Theorem 2.5 and Theorem 2.8 of \cite{AlexHGGA}]\label{modelExistenciallyClosed} Let $S$ be a countable and amenable group. Then the model $(\infty \ell_2(S),\infty \lambda_S) := \oplus_{n\geq 1}(\ell_2(S),\lambda_S)$ (countable copies of the representation $(\ell_2(S),\lambda_S)$) is existentially closed  and its theory has quantifier elimination.

\end{fact}

The theory of the model described in Fact \ref{modelExistenciallyClosed} is the \textit{model companion} of the theory of a Hilbert space expanded by any unitary representation of $S$. The proof provided in \cite{AlexHGGA} uses Hulanicki's theorem. In this paper, we will give a different proof of this result in Corollary \ref{existclosed} when the underlying group $G$ is finite.\\

Let us return to tools from representation theory. Let $T$ be a linear transformation over $V$, and let be $B$ a basis of $V$. If $[a_{ij}]_B$ is the matrix representation of $T$ in the basis $B$, then the \textit{trace} Tr$(T) := \sum_{i=1}^n a_{ii}$, is independent of the choice of $B$.

\begin{defn}[Definition 2.1 of \cite{pierre}]

Let $\pi$ be a linear representation of $G$ in $V$. For each $g \in G$, the map $\pi(g)$ is a linear transformation over $V$ and we denote the trace of $\pi(g)$ by $\chi_{\pi}(g) :=$ Tr$(\pi(g))$. This complex valued function $\chi_{\pi} : G \longrightarrow \Ce$ is called the \textit{character} of $\pi$.

\end{defn}

Characters  are important in representation theory since they determine the irreducible representations. Indeed, two representations having the same characters  are isomorphic (see Corollary 2, Theorem 4 of \cite{pierre}), meaning that there is a bijective linear transformation between the representations preserving the action of $G$. Additionally, in our context, characters play an important role since we will use them (see Fact \ref{fact:projectionsPierre} below) for defining projections from a linear representation onto its irreducible representations.\\

Given the group $G$, by Fact 
\ref{decompositionRegularRepresentation} that there are finitely many irreducible unitary representations $W_1, \ldots, W_k$ (modulo isomorphism) of $G$. 
Let $\chi_1, \ldots, \chi_k$ be their characters and let $n_1, \ldots, n_k$ be their degrees. 

Let $V$ be a finite dimensional vector space with a linear representation of $G$. Write $V = U_1 \oplus \ldots \oplus U_m$ as a decomposition of $V$ into a direct sum of irreducible representations of $G$. For each $i=1,\ldots, k$ we denote by $V_i$ the direct sum of those irreducible pieces among $U_1, \ldots, U_m$ that are isomorphic to $W_i$. Then, we can write $V = V_1 \oplus \ldots \oplus V_k$, a new decomposition of $V$ into sums of irreducible subrepresentations of $V$ that belong to distinct classes of isomorphisms.

\begin{fact}[Theorem 8 of \cite{pierre}]\label{fact:projectionsPierre}\

\begin{enumerate}

\item The decomposition $V \cong V_1\oplus \dots\oplus V_k$ does not depend on the initially chosen decomposition of $V$ into irreducible representations of $G$ in $V$.  

\item If $1 \leq i \leq k$, the projection $p_i$ of $V$ onto $V_i$ associated to the decomposition in (1) is given by $p_i = \frac{n_i}{|G|} \sum_{g \in G} \chi_i(g)^*\pi(g)$ (and may be identically $0$ when $W_i$ 
is not represented in $\pi$).

\end{enumerate}

\end{fact}

\begin{rem}\label{rem:projectionsPierre}

Consider now an infinite dimensional Hilbert space $H$ and a representation $\pi$ of $G$ in $H$ and consider the expansion of $H$ in the language $\mathcal{L}_{\pi}$ that includes a symbol for each $\pi(g)$ for $g\in G$. For each $1\leq i\leq k$ we let $P_i = \frac{n_i}{|G|} \sum_{g \in G} \chi_i(g)^*\pi(g)$, then the function $P_i$ is definable in $\mathcal{L}_{\pi}$.

\end{rem}

\subsection{Operator theory and C$^*$-algebras}
Let $H$ be a Hilbert space and let $T$ be a bounded linear operator on $H$.

\begin{defn}\label{Defn:spectrum}
The \textit{spectrum} of a linear operator $T$, denoted by $\sigma\!\parc{T}$, is defined as the set
\[
\sigma\!\parc{T} = \parl{\lambda \in \Ce : (T-\lambda I) \text{ is not bijective}}.
\]
The spectrum can be divided into three different types:

\begin{itemize}

\item $\sigma_p\!\parc{T} := \parl{\lambda \in \Ce : \ker(T-\lambda I) \not= 0}$; if $\lambda \in \sigma_p\!\parc{T}$ we call $\lambda$ a \emph{punctual eigenvalue} of $T$.
\item $\sigma_c\!\parc{T} := \Parl{\lambda\in \Ce :  \ker(T-\lambda I) = 0 \text{ and } \overline{Im(T-\lambda I)} = H}$; if $\lambda \in \sigma_c\!\parc{T}$ we call $\lambda$ an \emph{approximate eigenvalue} of $T$.
\item $\sigma_r\!\parc{T} := \Parl{\lambda \in \Ce :  \ker(T-\lambda I) = 0 \text{ and } \overline{Im(T-\lambda I)} \not= H}$; if $\lambda \in \sigma_r\!\parc{T}$ we call $\lambda$ a \emph{residual eigenvalue} of $T$.
%\item $\sigma_e\!\parc{T}$ is the set of accumulation points of $\sigma(T)$.
\end{itemize}

The \textit{punctual spectrum} is the collection of punctual eigenvalues, the \textit{continuous spectrum} is the collection of approximate eigenvalues, and the \textit{residual spectrum} is the collection of residual eigenvalues.

\end{defn}

\begin{fact}[Corollary 6.10.11 of \cite{naylor}]\label{Unitarioespectroresidual}

Let $T$ be a normal operator over a Hilbert space $H$. Then $T$ has no residual eigenvalues. Thus, the spectrum of a normal operator is divided only into two pieces
\[
\sigma(T) = \sigma_p\!\parc{T}\cup\sigma_c\!\parc{T}.
\]

\end{fact}

%\begin{defn}[II.4 \cite{ken}]
%Let $H$ and $K$ be separable Hilbert spaces and let $T : H \longrightarrow H$ and $S : K \longrightarrow K$ be normal operators. We say that $T$ and $S$ are \textit{approximately unitarily equivalent} if for all $\varepsilon > 0$ there exists a  linear surjective isometry
%\[
%\mathcal{O}_\varepsilon : K \longrightarrow H \text{ such that } \Norm{S - \mathcal{O}^*_\varepsilon T \mathcal{O}_\varepsilon} < \varepsilon.
%\]
%\end{defn}

%\begin{hecho}[Weyl-von Neumann-Berg Theorem, see Theorem II.4.4 \cite{ken}]\label{Neumann}Let $T$ and $S$ be normal operators over separable Hilbert spaces $H$ and $K$ respectively. Then $T$ and $S$ are approximately unitarily equivalent if and only if
%\begin{enumerate}

%\item $\sigma_e\!\parc{T}  = \sigma_e\!\parc{S},$

%\item $\ker\parc{T- \lambda I} = \ker\parc{S- \lambda I}$ for all $\lambda \in \Ce \setminus \sigma_e\!\parc{T}$ such that dim$\parc{\ker\parc{T- \lambda I}} < \infty$.
%\end{enumerate}
%\end{hecho}

Next, we introduce the concept of a C$^*$-algebra along with the concept of a representation of a C$^*$-algebra. These two are fundamental to understand the key concept that we will use in this paper: approximately unitarily equivalence between algebras.

\begin{defn}[Basics of \cite{ken}]

A \textit{Banach algebra} $\mathcal{A}$ is a complex normed algebra which is complete (as a topological space) and satisfies $\norm{AB} \leq \norm{A}\norm{B}$ for all $A, B \in \mathcal{A}$.

\end{defn}

\begin{defn}[Basics of \cite{ken}]

A $C^*$\textit{-algebra} $\mathcal{A}$ is Banach $*$-algebra (a Banach algebra with an involutive operation $*$) with the additional condition that  $\norm{A^*A} = \norm{A}^2$ for all $A \in \mathcal{A}$.

\end{defn}

\begin{example}[Example 1.1 of \cite{ken}]

The algebra of all bounded operators $B(H)$ on a Hilbert space $H$ is a C$^*$-algebra with the usual operation of adjoint $-^*$. This result follows from the equality:
\[
\norm{A^*A} = \sup_{\norm{x}=\norm{y}=1} |\gen{A^*Ax,y}| = \sup_{\norm{x}=\norm{y}=1} |\gen{Ax,Ay}| = \norm{A}^2.
\]

\end{example}

\begin{example}

The subalgebra of compact operators in $B(H)$ on a Hilbert space $H$, denoted by $K(H)$, also defines a C$^*$-algebra.

\end{example}

\begin{defn}[Basics of \cite{ken}]

Let $\mathcal{A}$ be a C$^*$-algebra and let $H$ be a Hilbert space. A map $\pi : \mathcal{A} \longrightarrow B(H)$ is said to be a $*$-\textit{representation} of $\mathcal{A}$ if $\pi$ is a homomorphism of $*$-algebras which commutes with the involution.

\end{defn}

\begin{defn}[Section II.4 of \cite{ken}]\label{def:repAUE}

Let $\mathcal{A}$ be a C$^*$-algebra and let $\pi_1$ and $\pi_2$ be representations of $\mathcal{A}$ on Hilbert spaces $H_1$ and $H_2$ respectively. The representations $\pi_1$ and $\pi_2$ are called \textit{Approximately Unitarily Equivalent} (AUE) if there is a sequence of unitary operators $\parl{\mathcal{O}_k}_{k \in \Na}$ with $\mathcal{O}_k:H_1\to H_2$ such that
\[
\pi_2(A) = \lim_{k \to \infty} \mathcal{O}_k \pi_1(A) \mathcal{O}^*_k \text{ for all } A \in \mathcal{A}. 
\]

where convergence is in the sense of the operator norm topology.

\end{defn}

%\begin{rem}
%This last definition includes the concept of a C$^*$-algebras and a $*$-representation of the algebra being approximately unitarily equivalent. Let $\mathcal{A}$ be a C$^*$-algebra, let $H$ be a Hilbert space, and let $\pi : \mathcal{A} \longrightarrow B(H)$ be a $*$-representation of $\mathcal{A}$, whose image we denote by $\mathcal{B}$. We say that  $\mathcal{A}$ and $\mathcal{B}$ are \textit{approximately unitarily equivalent} if Id$_\mathcal{A}$ and $\pi$ are approximately unitarily equivalent as representations.
%\end{rem}

There is a strong connection between two families of operators being approximately unitarily equivalent and the structures (the Hilbert spaces expanded with the C$^*$-algebras) being elementary equivalent. 
%This is an easy generalization a result by C.W. Henson (the case $m=1$).

\begin{rem}\label{JiontlyAproximatelyEquivalentFamilies}

Let $\mathcal{A}$ be a C$^*$-algebra and let $\pi_1$ and $\pi_2$ be representations of $\mathcal{A}$ on separable infinite Hilbert spaces $H_1$ and $H_2$ respectively. Assume the representations $\pi_1$ and $\pi_2$ are AUE. Then we have $
(H_1,\pi_1)\equiv
(H_2,\pi_2)$. 
\end{rem}

\begin{proof}
Since $H_1$ and $H_2$ are separable, we may assume $H_1=H_2$. Since the representations $\pi_1$ and $\pi_2$ are AUE there is a sequence $\parl{\mathcal{O}_k}_{k \in \Na}$ of unitary operators satisfying
\begin{equation}\label{eq.isoUltraPowers}
\pi_2(A) = \lim_{k \to \infty} \mathcal{O}_k \pi_1(A) \mathcal{O}^*_k \text{ for all } A \in \mathcal{A}. 
\end{equation}

Let $\mathcal{F}$ be a non-principal ultrafilter over $\mathbb{N}$ and consider the ultrapowers 
\[
\Pi_{k,\mathcal{F}}(H_1,\pi_1)  \text{ and }
\Pi_{k,\mathcal{F}}(H_2,\pi_2).
\]
 First define $\Phi\colon\Pi_{k,\mathcal{F}}H_1 \longrightarrow 
\Pi_{k,\mathcal{F}}H_2$ as the map $\Phi([(v_k)_k]_\mathcal{F})=[(v_k)_k]_\mathcal{F}$ induced by the identification $H_1=H_2$ as Hilbert spaces. We extend the function $\Phi$ to the maps in the representation by defining, for each $A\in \mathcal{A}$, the correspondence $\Phi [(\pi_1(A))_k]_\mathcal{F}=[(\mathcal{O}_k\pi_1(A)\mathcal{O}^*_k)_k]_\mathcal{F}$. For a fixed index $k\in \mathbb{N}$, the map $\mathcal{O}_k$ is a unitary transformation so the $k$-th component of the map $\Phi$ that sends $\pi_1(A)$ to $\mathcal{O}_k\pi_1(A)\mathcal{O}^*_k$ is an isomorphism of representations of $\mathcal{A}$. Moreover, by Equation \ref{eq.isoUltraPowers}, for each $A \in \mathcal{A}$ we have $\pi_2(A) = \lim_{k \to \infty} \mathcal{O}_k \pi_1(A)\mathcal{O}^*_k$, so the representation $\Pi_{k,\mathcal{F}}(H_2,\pi_2)$ is isomorphic to $\Pi_{k,\mathcal{F}}(H_2,\mathcal{O}_k\pi_1 \mathcal{O}^*_k)$ and thus $\Pi_{k,\mathcal{F}}(H_1,\pi_1) \equiv
\Pi_{k,\mathcal{F}}(H_2,\pi_2)$.

\end{proof}

Voiculescu's Theorem states that two separable C$^*$-algebras, say $\mathcal{A}$ and $\mathcal{B}$, are  AUE if there is a completely positive definite map $\Phi \colon \mathcal{A}\longrightarrow \mathcal{B}$ (see page 65 in \cite{ken}) such that $K(H)\cap \mathcal{A}\subseteq \ker(\Phi)$ (see Theorem II.5.3 of \cite{ken}). This last requirement is problematic in our context. Since we are considering  not only the regular representations (which have no compact operators in the generated C$^*$-algebra when the group is infinite), but arbitrary representations of discrete groups, it is likely that they may include compact operators. Thus, requiring the condition $\Phi(K(H)\cap\mathcal{A}) = 0$ is too restrictive. However, there are some consequences of Voiculescu's theorem which include a finer control on the behavior of compact operators under $\Phi$ that will work better in our setting.

\begin{defn}[Page 34 \cite{ken}]
Let $\mathcal{A}$ be a $C^*$-algebra and let $(H,\pi)$ be a representation of $\mathcal{A}$. We say that $\pi$
is non-degenerate if $\pi(\mathcal{A})H$ is dense in $H$.
\end{defn} 

\begin{fact}[Lemma II.5.7 of \cite{ken}]\label{fact:unitarilyequivalentcompact}

Let $\mathcal{A}$ be a $C^*$-subalgebra of the algebra of compact operators and let $\pi_1$ and $\pi_2$ non-degenerate representations of $\mathcal{A}$ on separable Hilbert spaces $H_1$ and $H_2$ respectively. Then, $\pi_1$ and $\pi_2$ are unitarily equivalent, i.e. there exists $U : H_1 \longrightarrow H_2$ unitary such that 
\[
\norm{\pi_2(A)-U\pi_1(A)U^*} = 0
\]
if and only if rank$(\pi_1(A)) = $rank$(\pi_2(A))$ for all $A \in \mathcal{A}$.

\end{fact}

\begin{fact}[Theorem II.5.8 of \cite{ken}]\label{fact:unitarilyapproximatelyequivalentwithcompacts}

Let $\mathcal{A}$ a separable C$^*$-algebra, and let $\pi_1$ and $\pi_2$ non-degenerate representations of $\mathcal{A}$ on separable Hilbert spaces. Then,  $\pi_1$ and $\pi_2$ are AUE, if and only if  rank$(\pi_1(A)) = $rank$(\pi_2(A))$ for all $A \in \mathcal{A}$.

\end{fact}

Model theoretic consequences of the previous fact go back to 
unpublished work by C. Ward Henson, who pointed out to the first author of the paper that this result characterizes, up to elementary equivalence, expansions of Hilbert spaces with a self adjoint operator (in the form of the Weyl-von Neumann-Berg theorem).

\begin{lem}\label{lem:compactnoncompact}

Let $\mathcal{A}$ be a separable C$^*$-algebra. Then, every $*$-representation $\pi$ of $\mathcal{A}$ on a separable Hilbert space $H$, can be written as $\pi_c \oplus \pi_c^\perp$, where $\pi_c$ is composed only by compact operators and $\pi^\perp_c$ has no compact operators. Moreover, we can write $H = H_c\oplus H_c^\perp$, and the operators in $\pi_c$ and $\pi_c^\perp$ act on $H_c$ and $ H_c^\perp$ respectively.

\end{lem}

\begin{proof}
    It folows from the proof Corollary II.5.9 of \cite{ken}. 
\end{proof}

We now introduce perturbations in our setting. 
The reader may want to check \cite{BY-perturbations} for an introduction to the subject. The reader may also want to check \cite{itai} for the theory of perturbations applied to expansions of Hilbert spaces with a single unitary map.

\begin{defn}\label{defn:pertubations}
Let $G$ be a countable group, let $\pi$ be a unitary representation of $G$ in an infinite dimensional Hilbert space $H$, and let $IHS_\pi=Th(H,\pi)$. Fix $\parl{g_n}_{n \in \Na}$ an enumeration of $G$. Let $(H_1, \pi_1)$ and $(H_2, \pi_2)$ be models of $IHS_\pi$ of the same density character and let $\varepsilon\geq 0$. We define an \textit{$\varepsilon$-perturbation} between $(H_1,\pi_1)$ and $(H_2,\pi_2)$ to be an isometric isomorphism of Hilbert spaces $U:H_1\cong H_2$ which also satisfies
\[
\sum_{n\geq 0}\frac{1}{2^n}\| U \pi_1(g_n) U^{-1}-\pi_2(g_n)\|\leq \varepsilon.
\]
The set of all $\varepsilon$-perturbations will be denoted by $\text{Pert}_\varepsilon((H_1,\pi_1),(H_2,\pi_2))$ or simply by $\text{Pert}_\varepsilon(H_1,\pi_1)$  if $(H_2,\pi_2)=(H_1,\pi_1)$.

Assume now that $(H_1,\pi_1)$ is saturated and strongly homogeneous and let $A\subset H_1$ be small and $p,q\in S_n(A)$. Given $\varepsilon\geq 0$, we say that $d_{pert}(p,q)\leq \varepsilon$ if there is a unitary map $U$ on $H_1$ and realizations $\vec a\models p$, $ \vec b\models q$ in $H_1$ such that $U\in \text{Pert}_\varepsilon(H_1,\pi_1)$, $U\! \upharpoonright_A=id_A$ and $\|U(\vec a)-\vec b\|\leq \varepsilon$. We say $IHS_\pi$ is $\aleph_0$-stable up to perturbations if for any $A\subset H_1$ separable, 
    the density character of $(S_1(A),d_{pert})$ is countable.

We say that $(H_1, \pi_1)$ is $\aleph_0$-saturated over $(H_0, \pi_0)$ up to perturbations if for any $(H_2, \pi_2)\succeq (H_0,\pi_0)$ with $\dim(H_2\cap H_1^{\perp}) \leq \aleph_0$ and every $\varepsilon>0$ there is an $\varepsilon$\textit{-perturbation} between $(H_2,\pi_2)$ and an elementary substructure of $(H_1,\pi_1)$ that fixes $(H_0,\pi_0)$ pointwise. 

We say the models $(H_1, \pi_1)$ and $(H_2, \pi_2)$ are \textit{approximately isomorphic} if for each $\varepsilon > 0$ there exists an $\varepsilon$-peturbation between them. The theory $IHS_\pi$ is called $\Ao$\textit{-categorical up to perturbations} if each pair of separable models of $IHS_\pi$ are approximately isomorphic.

%We say $IHS_\pi$ is $\aleph_0$-stable modulo perturbations if for every $(H_0, \pi_0)\models IHS_\pi$ with $\dim(H_0)=\aleph_0$ there is  $(H_1,\pi_1)\succeq (H_0,\pi_0)$ with $\dim(H_1)=\aleph_0$ such that $(H_1,\pi_1)$ is $\aleph_0$-saturated over $(H_0,\pi_0)$. Notice that being approximately isomorphic is a transitive property.

%We say that $IHS_\pi$ has $\aleph_0$-saturated up to perturbations if given $(H_0,\pi_0)\models IHS_\pi$, $A\subseteq H_0$, and, after adding constants for $A$ in the language $\mathcal{L}_\pi$, there is  $(H_1,\pi_1)\models IHS_{\pi,A}$ such that for all  $(H_2,\pi_2)\models IHS_{\pi,A}$ and all $\varepsilon\geq 0$ there is $(H_1',\pi_1')\preceq (H_2,\pi_2)$ with $A\subseteq H_1'$ and $U\in \text{Pert}_\varepsilon((H_1,\pi_1),(H_1',\pi_1'))$ such that $U\! \upharpoonright_A=id_A$ pointwise.

\end{defn}
    
\begin{rem}\label{rem:satimpliesstab}
Fix a theory $IHS_\pi$. Assume that for every separable $(H_0,\pi_0)\models IHS_\pi$ there is a separable $(H_1,\pi_1)\succeq (H_0,\pi_0)$ which is $\aleph_0$-saturated over $(H_0,\pi_0)$ up to perturbations. We claim that if this is the case, then $IHS_\pi$ is 
$\aleph_0$-stable up to perturbations. Indeed, by Lowenheim-Skolem any separable subset $A$ of a model of $IHS_\pi$ is a subset of a separable model $(H_0,\pi_0)\models IHS_\pi$. Then, if $(H_1,\pi_1)\succeq (H_0,\pi_0)$ is separable and $\aleph_0$-saturated over $(H_0,\pi_0)$, the density character of $(S_1(H_0),d_{pert})$ is countable and thus so is the case for $(S_1(A),d_{pert})$.
\end{rem}

\section{Hilbert spaces expanded by a representation of a finite group $G$}\label{rep:finite}

In this section we will study unitary representations of a finite group $G$ on a separable infinite dimensional Hilbert space $H$. Let $\mathcal{L}_{\pi}$ be as in Definition \ref{theoryoftheexpansion}. Let $\vec a\in H^n$, we will write $\tp(\vec a)$ for the type in the sense of Hilbert spaces and $\tp_\pi(\vec a)$ for the type in the extended language $\mathcal{L}_{\pi}$. Similarly, we write $\qftp_\pi(\vec a)$ for the quantifier free type of the tuple in the extended language.

We will show that for any such representation $\pi$, the theory $IHS_\pi=Th(H,\pi(g):g\in G)$ is $\Ao$-categorical, has quantifier elimination and is $\Ao$-stable.

Since $G$ is finite, by Fact \ref{decompositionRegularRepresentation} there are finitely many irreducible representations of $G$ and all of them are finite dimensional. Let $W_1,\dots,W_k$ be a list of these representations and, after reorganizing the list if necessary, we may assume there is $m\leq k$ such that $W_1, \ldots, W_{m}$ are the irreducible representations of $G$ having infinitely many copies in $(H,\pi(g):g\in G)$. Notice that, since $H$ is infinite dimensional and separable, the number $m$ can not be zero, and the number of copies of any $W_i$ for $1\leq i\leq k$ is at most $\Ao$. We can also define $H^{\text{fin}} = V_{m+1} \oplus \ldots \oplus V_{k}$ as the direct sum of the remaining irreducible representations, the ones appearing only finite many times in $H$, say with multiplicity $d_{m+1}, \ldots, d_{k}$ respectively. Then, if $m+1\leq i \leq k$, we have $V_i\cong W_i^{d_i}$. Notice that $k$ could be equal to $m$ and thus $H^{\text{fin}}$ could be the zero subspace. Finally write 
\[
H \cong \bigoplus_{t=1}^\infty W_1\bigoplus  \dots \bigoplus_{t=1}^\infty W_m \bigoplus H^{\text{fin}} 
\]

%From now on, we will denote by $\pi$ a unitary representation of $G$ in a infinite-dimensional separable Hilbert space $H$.

\begin{defn}\label{def:unitaryRepresentationGFinite}

For each $m+1\leq i \leq k$ we call $V_i = \bigoplus_{t=1}^{d_i} W_i $ a \emph{finite component} of $H$. On the other hand, for each $1\leq  i \leq m$ we call $H_i=\bigoplus_{t=1}^{\infty} W_i$ an \emph{infinite component} of $H$. Finally by a \emph{component} we mean a finite component or an infinite component. Also, by $H^{\text{inf}}$ we mean the sum $H_1\oplus \ldots \oplus H_m$.

%, we can define $\rho$ as the restriction of $\pi$ to $H^{\text{inf}}$. 
\end{defn}

\begin{prop}\label{ProyeccionesPierreSonDefinibles}

The projections from $H$ onto each of its components are definable in the theory $IHS_\pi$.

\end{prop}

\begin{proof}

Fix $1 \leq i_0 \leq k$ and take $n_{i_0}$ and $\chi_{i_0}(g)$ as in Fact \ref{fact:projectionsPierre}(2). Then (see Remark \ref{rem:projectionsPierre}) the function $P_{i_0}\colon H \longrightarrow H$, defined as 
\begin{equation}\label{eq:projectiononH}
P_{i_0}v=\frac{n_{i_0}}{|G|} \sum_{g \in G}\chi_{i_0}(g)^*\pi(g)v
\end{equation}
is a definable function in the language $\mathcal{L}_\pi$. Given $v\in H$, we can write $v$ as the sum 
\[
\sum_{t= 1}^\infty \sum_{1\leq i \leq m} v_i^{(t)} + \sum_{m+1\leq i \leq k} v_i \in H^{\text{inf}}\oplus H^{\text{fin}},
\]
where for $1\leq i \leq m$ the vector $v_i^{(t)}$ denotes the projection of $v$ on the $t$-th copy of $W_i$ in $H^{\text{inf}}$ and for $m+1 \leq i \leq k$ the vector $v_{i}$ denotes de projection of $v$ on $V_i$ in $H^{\text{fin}}$. We will prove that the equation defined in \ref{eq:projectiononH} defines the projection onto $H_{i_0}$. We will do the proof for some infinite component $H_{i_0}$ of $H$, the proof for a finite component of $H$ is analogous. By definition
\[
P_{i_0}v = \dfrac{n_{i_0}}{|G|} \Sum{g \in G}{}\chi_{i_0}(g)^* \pi(g)(\sum_{t= 1}^\infty \sum_{1\leq i \leq m} v_i^{(t)} + \sum_{m+1\leq i \leq k} v_i) = 
\]
\[
 \sum_{t = 1}^\infty \sum_{1\leq i \leq m} \dfrac{n_{i_0}}{|G|} \Sum{g \in G}{}\chi_{i_0}(g)^* \pi(g)  v_i^{(t)} + \sum_{m+1 \leq i \leq k} \dfrac{n_{i_0}}{|G|} \Sum{g \in G}{}\chi_{i_0}(g)^* \pi(g) v_i.
\]
Notice that we can restrict $\pi(g)$ to each copy of the irreducible subrepresentations $W_i$ of $\pi$
\[
\pi(g) v_i^{(t)} = \pi(g)\!\!\upharpoonright\!\!_{W^{(t)}_{i}}  v_i^{(t)} \text{ and }  \pi(g) v_i = \pi(g)\!\!\upharpoonright\!\!_{W_{i}} v_i,
\]
then using the projection $p_{i_0}$ given by Fact \ref{fact:projectionsPierre}(2), we could write $P_{i_0}v$ as
\[
P_{i_0}v = \sum_{t = 1}^\infty\sum_{1 \leq i \leq m} p_{i_0} v_i  + \sum_{m+1 \leq i \leq k} p_{i_0} v_i =  \sum_{t= 1}^\infty v_{i_0}^{(t)}.
\]
We can conclude that the function $P_{i_0}$ is the projection of $H$ onto the infinite component $H_{i_0}$ of $H$ and it is definable.
\end{proof}

\begin{thm}\label{GFinitioAocategoricidad}
Let $\pi$ be a unitary representation of $G$ on a Hilbert space $H$ such that $H = H^{\text{inf}}$. Then the theory $IHS_\pi$ is $\Ao$-categorical.
\end{thm}

\begin{proof}
First note that the theory $IHS_\pi$ includes as part of the information that $(\pi(g):g\in G)$ is a representation of $G$ by unitary maps.

By Proposition \ref{ProyeccionesPierreSonDefinibles}, for each $1\leq i \leq k$ the projection $P_i$ onto $H_i$ is definable in the theory $IHS_\pi$. Fix some $1\leq i \leq m$, and consider the following scheme indexed by $n$. To simplify the notation, the indexes $t, s$ will range over $\parl{1, \ldots, n}$ and the indexes $j_1, j_2$ and $j$ will range over $\parl{1, \ldots, n_i}$ (recall that $n_i$ is the degre of $W_i$):
\begin{equation}\label{EsquemaProyeccionesPierre}
\inf_{\overline{v}_t} \max \Parl{\max_{t,j} \Parl{\big|\norm{v^j_t}-1\big|, \norm{P_iv^j_t - v^j_t}}, \max_{t < s,j_1 \leq j_2}\big|\gen{v^{j_1}_t,v^{j_2}_s}\big|, \max_{t,j_1 < j_2}\big|\gen{v^{j_1}_t,v^{j_2}_t} \big|} = 0.
\end{equation}
In the sentence above $\overline{v}_t$ denotes the set of vectors $\Parl{v^1_t, \dots, v^{n_i}_t}$. The $n$-th sentence in the Scheme \ref{EsquemaProyeccionesPierre} indicates that there are $n$ collections of $n_i$ vectors $\Parl{v^1_t, \dots, v^{n_i}_t}_{t=1}^n$ which are almost orthonormal and almost invariant under the projection $P_i$. Observe that this scheme belongs to $IHS_\pi$ for any $1\leq i\leq m$, that is, for any of the irreducible representations $W_i$ of $G$ appearing in $H$. Additionally, the sentence
\begin{equation}\label{eq:projectionsZero}
\sup_{v}\max_{m+1\leq i\leq k}\|P_i(v)\|=0,
\end{equation}
indicates that the irreducible representation $W_i$ for $i\geq m+1$ are not represented in $(H,\pi)$.

Now let $(K,\rho)\models IHS_\pi$ be separable. Since $K$ is a model of the theory \textit{IHS}, we have that for all $n \in \Na$ there is $\varepsilon > 0$, such that for $j_1 \leq j_2$ and $j$, and $s < t$ as above if there are vectors in $K$ satisfying \[
|1-\norm{v^{j}_t}_2|< \varepsilon \text{ and } |\langle v^{j_1}_t,v^{j_2}_s\rangle| < \varepsilon
\]
then (after applying Gram-Schmidt and taking a different family) there is a family of exact witnesses for the property described in the equation. Thus we may assume the set $\Parl{v^{1}_t, \ldots, v^{n_i}_t}$ is orthonormal to the set $\Parl{v^{1}_s, \ldots, v^{n_i}_s}$ when $t\neq s$. Now, for all $n \in \Na$ and each $1\leq i \leq m$, we shall prove there are $n$ orthogonal copies of $W_i$ in $K$. For each $t$ consider the vectors 
\[
w^1_t := \frac{P_iv^1_t}{\norm{P_iv^1_t}}, \ldots, w^{n_i}_t := \frac{P_iv^{n_i}_t}{\norm{P_iv^{n_i}_t}}.
\]
It follows from Scheme \ref{EsquemaProyeccionesPierre}, that the set $\Parl{w^1_t, \dots, w^{n_i}_t}$ forms a basis for a copy of the irreducible representation $W_i$. Furthermore if $ s < t $ the set $\Parl{w^1_t, \ldots, w^{n_i}_t}$ is orthogonal to the set $\Parl{w^1_s, \ldots, w^{n_i}_s}$. Thus, for any irreducible representation $W_i$ appearing in $H^{\inf}$, actually we can find a countable number copies of it in $K$.\\

Moreover, by Proposition \ref{ProyeccionesPierreSonDefinibles} we can build a first order sentence axiomatizing that any vector $v \in K$ (or any other model of $IHS_\pi$) can be written as $v = \sum_{1\leq i \leq m} P_i v$. Therefore, since both $(H,\pi)$ and $(K,\rho)$ are separable, the multiplicity of each $W_i$ in $K$ is $\Ao$, and $K$ is only composed of copies of $W_i$ for $1\leq i \leq m$. Since both $K$ and $H$ contain $\Ao$ copies of each irreducible representation $W_1, \ldots, W_m$, and no copy of $W_i$ for $m+1\leq i\leq k$, the two representations are isomorphic.
\end{proof}

We can extend Theorem \ref{GFinitioAocategoricidad} to a general unitary representation $\pi$ of $G$ acting on a separable Hilbert space $H$:

\begin{cor}\label{cor:wCategoricityGFinite}
Let $\pi$ be a  representation of $G$ on an infinite dimensional Hilbert space. Then the theory $IHS_\pi$ is $\Ao$-categorical.

\end{cor}

\begin{proof}
Following the notation of Definition \ref{def:unitaryRepresentationGFinite} associated to the model $(H,\pi)$, we get
a new axiomatization by keeping Scheme \ref{EsquemaProyeccionesPierre} and replacing the sentence \ref{eq:projectionsZero} for a collections of sentences describing the dimension and multiplicity of each irreducible representation $W_{i}$ in $H^{\text{fin}}$. Fix some $m+1\leq i \leq k$, as in Theorem \ref{GFinitioAocategoricidad}, to simplify the notation, the indexes $s, t$ will range over $\parl{1, \ldots, d_i}$, and the index $j$ will range over $\parl{1, \ldots, n_i}$. Consider the following sentence from $IHS_\pi$:
\begin{equation}\label{equation:SentenceIrreducibleFinMult}
\inf_{\overline{v}_t} \sup_{\norm{v}=1}\max\lbrace\max_{t, j} \Parl{\big|\norm{v^j_t}-1\big|, \norm{P_iv^j_t - v^j_t}},\max_{s<t,j}\big|\gen{v^j_t,v^j_s}\big|, \norm{P_iv-\sum_{t=1}^{d_i}\sum_{j=1}^{n_i}{\gen{v_t^j,P_iv}v_t^j}}\rbrace = 0.
\end{equation}
In the sentence, $\overline{v}_t$ denotes the set of vectors $\Parl{v^1_t, \dots, v^{n_i}_t}$, and recall from the beginning of the section that for each $1+m \leq  i \leq k$ the multiplicity of $W_i$ in $H^{\text{fin}}$ is $d_{i}$, and $n_i$ is the dimension of $W_{i}$. As we shall prove below, the sentence in \ref{equation:SentenceIrreducibleFinMult} axiomatizes the presence of exactly $d_i$ copies of $W_i$ in $H^{\text{fin}}$.\\

Let $(K,\rho)$ be a separable model of $IHS_\pi$, and let $\varepsilon > 0$. We can assume that for each $t$, the family $\parl{v^{1}_t,\ldots, v^{n_i}_t}$ given by the sentence 
(\ref{equation:SentenceIrreducibleFinMult}) satisfies
\[
|\norm{P_iv^1_t}-1| < \varepsilon,\ldots,|\norm{P_iv_t^{n_i}}-1| < \varepsilon.
\]
Furthermore, using that $K\models IHS$ is $\aleph_0$-saturated (or by choosing $\varepsilon$ sufficiently small), there is subspace $K_i$ of $K$ with dimension $n_id_i$ invariant under the action of $\pi$ with $P_i(K_i)=K_i$. Thus by Fact \ref{fact:projectionsPierre} there are at least $d_i$ different copies of $W_i$ in $K$. Moreover, that is the exact number of copies of $W_i$ in $K$. Suppose there is one more copy of $W_i$ with vector basis $w_1,\ldots, w_{n_i}$, then using sentence (\ref{equation:SentenceIrreducibleFinMult})  the projection of each $w_j$ in the orthogonal complement of $K_i$ is arbitrary small and thus it should be zero.

It also follows from the proof of Theorem \ref{GFinitioAocategoricidad} that $\dim(H_i)=\dim(K_i)=\Ao$ for all $1\leq i\leq m$, so we get $\dim(H_i)=\dim(K_i)$ for all $1\leq i\leq k$ and thus the structures are isomorphic.
\end{proof}

\begin{rem}\label{projectagree}
Let $(K,\rho)$ be a model of $IHS_\pi$. Let $C\subseteq K$ be a closed subspace. We denote by $\Pr_C$ the orthogonal projection from $H$ onto $C$. Let $\overline{a}$ and $\overline{b}$ be tuples of $K^n$, if for each $1 \leq i \leq k$ we write $P_i$ for the projection on the $i$-th component of $K$, then  
\[
\qftp_\pi(\overline{a}/C)=\qftp_\pi(\overline{b}/C) \text{ implies } \text{Pr}_C(P_i(a_j))=\text{Pr}_C(P_i(b_j)).
\]
\end{rem}

\begin{proof}
    By Proposition \ref{ProyeccionesPierreSonDefinibles} each $P_i$ is quantifier free definable, so for each $1 \leq i \leq k$ and each $1\leq j\leq n$ we have $\qftp_\pi(\overline{a}/C)=\qftp_\pi(\overline{b}/C)$ implies $\qftp(P_i(a_j)/C)=\qftp(P_i(b_j)/C)$. It follows from basic results on $IHS$ (see for example \cite[Lemma 15.1]{AlexMT}) that $\qftp(P_i(a_j)/C)=\qftp(P_i(b_j)/C)$ implies $\text{Pr}_C(P_i(a_j))=\text{Pr}_C(P_i(b_j))$.
\end{proof}

\begin{prop}\label{cuantifiereliminationGfinite}
Let $\pi$ be a unitary representation of a finite group $G$ on a infinite dimensional Hilbert space $K$. Then the theory $IHS_\pi$ has quantifier elimination.

\end{prop}

\begin{proof}
Let $(K,\pi)$ be a separable model of $IHS_\pi$. We will prove that if $\overline{a}$ and $\overline{b}$ are arbitrary $n$-tuples of $K$ such that $\qftp_\pi(\overline{a})=\qftp_\pi(\overline{b})$ then $\tp_\pi(\overline{a}) = \tp_\pi(\overline{b})$. Fix $1 \leq i_0 \leq k$ and define $K_{i_0} := P_{i_0}(K)$. Additionally, let us assume without loss of generality that $\overline{a}$ and $\overline{b}$ are non-trivial arbitrary $n$-tuples of $K_{i_0}$ and define 
\[
A_{i_0} := \bigcup_{g \in G} \Parl{\pi(g)a_1, \dots, \pi(g) a_{n}} \text{ and } B_{i_0} := \bigcup_{g \in G} \Parl{\pi(g)b_1, \dots, \pi(g) b_{n}}.
\]
Thus, by assumption $\gen{A_{i_0}}$ and $\gen{B_{i_0}}$ (the span of the corresponding sets) are non-trivial subspaces of $K_{i_0}$.

\begin{claim}

There is an $\mathcal{L}_\pi$-isomorphism sending $\gen{A_{i_0}}$ to $\gen{B_{i_0}}$. 

\end{claim}
\begin{proof}

Since $\gen{A_{i_0}}$ and $\gen{B_{i_0}}$ are subspaces of $K_{i_0}$ closed under the action of $\pi$, by Fact \ref{fact:projectionsPierre}(1) we can decompose these finite dimensional vector spaces as the direct sum of copies of the irreducible representation $W_{i_0}$. Notice that $\qftp_\pi(\overline{a})=\qftp_\pi(\overline{b})$ implies
\begin{equation}\label{eq:qe} 
\gen{\pi(h)a_i,\pi(g)a_j} = \gen{\pi(h)b_i,\pi(g)b_j} \text{ for all } 1\leq i,j\leq n \text{ and } g, h  \in G,
\end{equation}
and recall that the linear independence of one vector from others is a quantifier free sentence expressible in the language of Hilbert spaces. Then, these inner product equations imply that if we take a subset of $A_{i_0}$ which forms a bases for $\langle A_{i_0} \rangle$ then the corresponding subset in $B_{i_0}$ forms a bases of $\gen{B_{i_0}}$, hence dim$\gen{A_{i_0}}=$ dim$\gen{B_{i_0}}$. Then we can write
\[
\gen{A_{i_0}} \cong U^{A}_1 \oplus \ldots \oplus U^{A}_\ell \text{ and } \gen{B_{i_0}} \cong U^{B}_1 \oplus \ldots \oplus U^{B}_\ell,
\]
for some $1 \leq \ell$, where each term of the sum is a copy of $W_{i_0}$ where the sum for
$\gen{B_{i_0}}$ comes from the map
sending $A_{i_0}$ to $B_{i_0}$. In fact, for all $1 \leq j \leq n$ and all $1 \leq \ell' \leq \ell$ we have that 
\[
\norm{\text{Pr}_{U^{A}_{\ell'}}a_j} = \norm{\text{Pr}_{U^{B}_{\ell'}}b_j}
\]
is implied by $\qftp_\pi(\overline{a})=\qftp_\pi(\overline{b})$ and the way we constructed the sum. Now, take any of the copies of $W_{i_0}$ in $\gen{A_{i_0}}$, namely $U^{A}_{\ell'}$. As above, if we choose a basis of $U^{A}_{\ell'}$ from $\gen{A_{i_0}}$ its corresponding subset of $\gen{B_{i_0}}$ give us a basis for $U^{B}_{\ell'}$ and the action by the maps $(\pi(g):g\in G)$ is compatible with this correspondence. 

%Then, if we define $f_{\ell'} \colon U^{A}_{\ell'} \longrightarrow U^{B}_{\ell'}$ as the linear map induced by the natural assignment between these basis, this map satisfies that if $a_j \in U^{A}_{\ell'}$ then $f_{\ell'}(a_j) = b_j \in U^{B}_{\ell'}$. Thus, $\sum_{1 \leq \ell' \leq \ell} f_{\ell'} 

Then the bijection $\gen{A_{i_0}} \longrightarrow \gen{B_{i_0}}$ induced by the bijection 
$A_{i_0} \longrightarrow B_{i_0}$ is an $\mathcal{L}_\pi$-isomorphism between $\gen{A_{i_0}}$ and $\gen{B_{i_0}}$.
\end{proof}

Notice that, if we start with $\overline{a}$ and $\overline{b}$ as general tuples of $K$, we can apply, for each $i\leq k$, the above construction to obtain a map $f_i:\gen{P_i(\overline a)}\to \gen{P_i(\overline b)}$ that respects the action of $\pi$. Putting the $f_i$ together into a single function $f$ one constructs a $\mathcal{L}_\pi$-isomorphism $f$ between the closed subspaces 
\[
A := \gen{\cup_{g \in G} \Parl{\pi(g)a_1, \dots, \pi(g) a_{n}}} \text{ and } B := \gen{\cup_{g \in G} \Parl{\pi(g)b_1, \dots, \pi(g) b_{n}}}.
\]
sending $\overline{a}$ to $\overline{b}$. Now we extend $f$ to an $\mathcal{L}_{\pi}$ isomorphism $\phi$ over $K$. Observe that we can write $K \cong A \oplus A^{\perp}$ and $K \cong B \oplus B^{\perp}$. Moreover, both $A^{\perp}$ and $B^{\perp}$ are models of $T_{\pi'}$, where $\pi'$ is the restriction of $\pi$ to $A^{\perp}$, by Corollary \ref{cor:wCategoricityGFinite} we have $A^{\perp}\cong B^{\perp}$ as $\mathcal{L}_{\pi'}$-structures. Thus, there is a $\mathcal{L}_{\pi}$-isomorphism $\phi \colon K \longrightarrow K$ sending $\overline{a}$ to $\overline{b}$.
\end{proof}

\begin{cor}\label{existclosed}\
Consider the representation formed by the sum of countable copies of the left regular representation, denoted by $\infty \lambda_G$, which acts on the separable Hilbert space $\infty \ell_2(G)$. Then, for any unitary representation $\pi$, we can embed a separable model $(H,\pi)$ of $IHS_\pi$ into $(\infty \ell_2(G),\infty \lambda_G)$ and the structure $(\infty \ell_2(G),\infty \lambda_G)$ is existentially closed.

\end{cor}

\begin{proof}
Let $(H,\pi)$ be as in the hypothesis. By Fact \ref{decompositionRegularRepresentation} all irreducible representations of $G$ appear in $(\infty \ell_2(G),\infty \lambda_G)$ with infinite  multiplicity. In particular, each of the irreducible representations used in $(H,\pi)$ (with finite or infinite multiplicity) appears in $(\infty \ell_2(G),\infty \lambda_G)$ and thus we can embed the structure $(H,\pi)$ into $(\infty \ell_2(G),\infty \lambda_G)$. \\

Now we show the structure $(\infty \ell_2(G),\infty \lambda_G)$ is existentially closed. Assume that $(H,\pi)\geq (\infty \ell_2(G),\infty \lambda_G)$ is a separable superstructure. As we mentioned above, we can find a copy $(\infty \ell_2(G)',\infty \lambda_G')$ extending $(H,\pi)$ and thus also extending $(\infty \ell_2(G),\infty \lambda_G)$. By quantifier elimination,  $(\infty \ell_2(G),\infty \lambda_G) \preccurlyeq (\infty \ell_2(G)',\infty \lambda'_G)$ and thus all existential witnesses in $(H,\pi)$ have an approximate witness inside $\infty \ell_2(G)$.
\end{proof}

This gives another proof that $Th(\infty \ell_2(G),\infty \lambda_G)$ is the model companion of the theory of representations of $G$ (compare with Theorem 2.8 in \cite{AlexHGGA}).

\begin{thm}\label{stabilityGfinite}\
Let $\pi$ be a unitary representation of a finite group $G$ in a infinite dimensional Hilbert space. Then the theory $IHS_\pi$ is $\Ao$-stable.
\end{thm}

\begin{proof}

%by Corollary \ref{cor:wCategoricityGFinite} and following its notation, for all $1\leq m \leq k_1$ we can find an countable number of basis $\Parl{v^{1}_t, \dots v^{d_m}_t}_{t\in \Na}$, such that for each $t\geq1$ we have that $V_m^{(t)}:=\gen{v^{1}_t, \dots v^{d_m}_t}$ is a copy of $V_m$. Also, for all $1 \leq m \leq k_2$ we can find $n_m$ number of basis $\Parl{v^{1}_t, \dots v^{d_m}_t}_{t=1}^{n_m}$, such that for each $1 \leq t \leq n_m$ we have that $W_{m}^{(t)}:=\gen{w^{1}_t, \dots w^{d_m}_t}$ is a copy of the irreducible representation $W_m$. Then We also write $K_m$ for the sum $\bigoplus_{t = 1}^\infty V_m^{(t)}$. Since $Th(K,\pi)$ is $\aleph_0$-categorical, i

Let $\parc{K, \rho} $ be a separable model of $IHS_\pi$, by Corollary \ref{cor:wCategoricityGFinite} we can decompose $K = K^{\text{inf}}\oplus K^{\text{fin}}$ as in Definition \ref{def:unitaryRepresentationGFinite}, then
\[
K \cong (K_1 \oplus\dots \oplus K_m) \oplus (V_{m+1} \oplus\dots \oplus V_k).
\]
To show that the theory $IHS_\pi$ is $\Ao$-stable, it is sufficent to take $E \subseteq K$ countable such that $\overline{E} =K$ and prove that the density character $\|\parc{S_1(E), d}\| \leq \aleph_0$. Let us consider $(\hat{K}, \hat{\rho}) \succcurlyeq \parc{K, \rho}$ the superstructure defined as
\[
\hat{K} :=  K \bigoplus W_{1}^0 \oplus\dots \oplus W_{m}^0,
\]
where $W_{1}^0 \ldots W_{m}^0$ are copies of the irreducible representations $W_{1}, \ldots, W_{m}$ in $K^{\text{inf}}$ respectively and $\hat{\rho}$ is the direct sum of the homomorphism from $G$ to $W_{1}^0 \oplus\dots \oplus W_{m}^0$ with $\rho$.\\ 

Let $\parc{F, \tau} \succcurlyeq \parc{K, \rho}$ be an arbitrary separable elemental superstructure. By $\Ao$-categoricity of the theory $IHS_\pi$, we can write $F \cong F^{\text{inf}} \oplus F^{\text{fin}}$ and notice that $K^{\text{fin}} = F^{\text{fin}}$, then
\[
F \cong F_1 \oplus\dots \oplus F_m \bigoplus F^{\text{fin}}
\] 
and each of the spaces $F_i$ with $1\leq i\leq m$ has dimensions $\Ao$. Now, for each $1 \leq i \leq k$ let us define $\text{Pr}_{\overline{E}}^{i}:F\longrightarrow F_i\cap \overline{E}$, as the projection of $F$ onto $F_i\cap \overline{E}$ defined by $\text{Pr}_{\overline{E}}^{i} = $ Pr$_{\overline{E}}P_i$, where Pr$_{\overline{E}}$ is the orthogonal projection onto the closed subspace $\overline{E}$, and $P_i : F \longrightarrow F_i$ is the projection on the $i$-th component of $F$ (defined as in Proposition \ref{ProyeccionesPierreSonDefinibles}). Now, take $v \in F$, by Proposition \ref{cuantifiereliminationGfinite} the theory $IHS_\pi$ has quantifier elimination and the type $\tp_\pi(v/\overline{E})$ is determined by the elements $\text{Pr}_{\overline{E}}(v)$ and the types $\Parl{ \tp_\pi(P_iv - \text{Pr}_{\overline{E}}^{i}v) : 1 \leq i \leq m}$, i.e. the types of elements orthogonal to $\overline{E}$ that lie in the different components of $F^{\text{inf}}$. Since $\overline{E} = K$ and $K^{\text{fin}} = F^{\text{fin}}$, and for all $m+1 \leq i \leq k$ we have $\text{Pr}_{\overline{E}}^{i}v \in K_i$. Thus we only need to realize in $(\hat{K},\hat{\rho})$ the types $\Parl{\tp_\pi(P_iv - \text{Pr}_{\overline{E}}^{i}v) : 1 \leq i \leq m}$ in a space orthogonal to $\overline{E}$.\\

Fix an index $1 \leq i \leq m$ and set $w_i := P_{i}v-\text{Pr}_{\overline{E}}^{i}v$. If $w_i = 0$ there is nothing to prove. Otherwise $\gen{\cup_{g \in G}\parl{\tau(g)w_i}}$ is a copy of $W_i$. Then there exists an isomorphism
\[
f_i \colon \gen{\cup_{g \in G}\parl{\tau(g)w_i}} \longrightarrow W^0_{i}\subseteq \hat{K},
\]
respecting the action of the operators induced by $\tau$ and $\hat{\rho}$. Repeating the argument for all $1\leq i \leq m$, we define
\[
\hat{v} := (\text{Pr}_{\overline{E}}(v) + (f_1(w_{1}) +  \dots +  f_m(w_m)) \in \hat{K}.
\]
Then we have $\tp_\pi(\hat{v}/\overline{E}) = \tp_\pi(v/\overline{E})$ and the new realization of the type belongs to $(\hat{K}, \hat{\rho})$ which is separable, hence $\norm{(S_1(E),d)} \leq \Ao$.
\end{proof}

Now we characterize algebraic closure in models of $IHS_\pi$ and we will also give a natural description of non-forking.
We work in $(H, \pi)$ a $\kappa$-saturated and $\kappa$-strongly homogeneous model of $IHS_\pi$ for some uncountable inaccessible cardinal $\kappa$. We say that $A \subset H$ is \textit{small} if $|A| < \kappa$, and if $C\subseteq H$ is a closed subspace, we denote by $\text{Pr}_C$ the orthogonal projection of $H$ onto~$C$. Additionally, if $A \subseteq H$ we write acl$(A)$ and dcl$(A)$ for the algebraic and definable closures of $A$ in the language $\mathcal{L}_\pi$, respectively, and cl$(A)$ for the topological closure.

\begin{prop}

Let $A \subset H$ be small. Then $$\acl(A)=\cl(\gen{\{\pi(g)(a):a\in A,g\in G\} \cup H^{\text{fin}}})$$

\end{prop}

\begin{proof}
"$\supseteq$" Notice that, for each $m+1 \leq i \leq k$ the component $V_i$ of $H$ is finite-dimensional, and by Proposition \ref{ProyeccionesPierreSonDefinibles} the projection $P_i : H \longrightarrow V_i$ is definable in the theory $IHS_\pi$. Then, the unitary ball of $V_i$ is definable over $\emptyset$ and compact and thus  algebraic over $\emptyset$. Hence $H^{\text{fin}} \subseteq \acl(A)$. Clearly we also have $\{\pi(g)(a):a\in A\}\subseteq \acl(A)$ and thus the containment follows.

"$\subseteq$" Now, set $E := \text{cl}(\gen{\{\pi(a):a\in A, g\in G\} \cup H^{\text{fin}}})$ and suppose $v \not\in E$, then $\Norm{\text{Pr}_{E^\perp}(v)}>0$. By hypothesis $E^\perp \subset H^{\text{inf}}$, so there exists $1\leq j \leq m$ such that $\Norm{P_{j}(\text{Pr}_{E^\perp}v)}>0$. Since $A$ is a small subset of $H$, so is $E$. Since $H_j$ is large, the subspace $H_j \cap E^\perp$ has infinite dimension. Thus, we can find a sequence $\Parl{v^j_t}_{t=1}^\infty \subset H_j \cap E^\perp$ of orthogonal vectors with norm $\Norm{P_{j}(\text{Pr}_{E^\perp}v)}$, such that 
\[
\tp(v^j_t +\sum_{i=1, i\not=j}^m P_i(\text{Pr}_{E^\perp}v) +\text{Pr}_Ev/A) = \tp(v/A) \text{ for all } t \geq 1.
\]
The elements of the sequence $\Parl{v^j_t +\sum_{i=1, i\not=j}^m P_i(\text{Pr}_{E^\perp}v) +\text{Pr}_Ev}_{t=1}^\infty$ are at the same positive distance one from the other
and thus $v\not \in \acl(A)$
\end{proof}

\begin{obse}
It follows from the previous proof that for $A \subset H$ be small we have $\dcl(A)=\cl(\gen{\{\pi(g)(a):a\in A,g\in G\}})$.    
\end{obse}

For the rest of this section, whenever $A\subset H$, we write $\overline{A}$ for the algebraic closure of $A$ in the language $\mathcal{L}_\pi$.
To deal with non-forking, we introduce an abstract notion of independence and then show it coincides with non-forking.
Our approach follows the argument for a single unitary operator presented in \cite{Argoty}. A similar characterization was  used in \cite[Section 3]{AlexHGGA} to show that that $IHS_\pi$ is superstable when $G$ is countable and the expansion is existentially closed. Instead of repeating again all of the steps, we will prove the key steps that make the arguments work.

\begin{defn}\label{defn:independence}
Let $(H, \pi)\models IHS_\pi$ be $\kappa$-saturated and $\kappa$-strongly homogeneous.
Let $\vec{a} = (a_1,\dots, a_n) \in H^n,$ and let $A,B $ and $C \subset H$ be small. We say that $\vec{a}$ is $\ast$-\textit{independent} \textit{from} $B$ \textit{over} $C$ if for all $1 \leq j\leq n$ and $1\leq i \leq k$ we have $\text{Pr}_{\overline{B\cup C}}(P_i(a_j)) = \text{Pr}_{\overline{C}}(P_i(a_j))$. If $\vec{a}$ is $\ast$-independent from $B$ over $C$ we write $\vec{a}\forkindep[C]B$. Additionally, if all finite subsets $\vec{a}$ of $A$ are such that $\vec{a}\forkindep[C]B$, we say that $A$ is $\ast$-\textit{independent from} $B$ \textit{over} $C$ and we write $A\forkindep[C]B$.

\end{defn}

\begin{lem}\label{CommtingLemmaForking}

Let $C\subseteq H$ be such that $C=\overline{C}$ and let $v \in H$. Then, for all $1\leq i \leq k$ we have $P_i(\text{Pr}_C(v)) = \text{Pr}_C(P_i(v))$.

\end{lem}

\begin{proof}

Notice that for all $v \in H$ we can write $v = v_1 + v_2$, where $v_1 \in C$ and $v_2 \in C^\perp$. Let $g\in G$, then for all $c \in C$ we have
\[
\pi(g)c \in C \text{ and } \gen{\pi(g)c, \pi(g)v_2} = \gen{c, v_2} = 0.
\]
Since $\pi(g)$ acts in $C$ as a bijection, we obtain $\pi(g)v_2 \in C^\perp$. Thus, 
\[
\text{Pr}_C\pi(g)v = \pi(g)v_1 =  \pi(g)\text{Pr}_Cv,
\]
this means that $\text{Pr}_C$ and $\pi(g)$ commute for all $g \in G$. Since the projection $P_i$ is a linear combination of the operators $\Parl{\pi(g)}_{g \in G}$ (see Proposition \ref{ProyeccionesPierreSonDefinibles}), it follows that $\text{Pr}_C$ commutes with the projection $P_i$ for all $1\leq i \leq k$.

\end{proof}

The previous result will allow us to characterize independence over closed sets without  using the projections $P_i$ over the components, just as was done in \cite[section 3]{AlexHGGA}:

\begin{cor}\label{forkingEquivalenceGFinite}

Let $\vec{a} = (a_1,\dots, a_n) \in H^n,$ and let $B, C \subset H$ be small. Then, $\vec{a}\forkindep[C]B$ if and only if for each $1\leq j \leq n$, we have $ \text{Pr}_{\overline{B\cup C}}a_j = \text{Pr}_{\overline{C}}a_j$.

\end{cor}

\begin{proof}

Suppose that $\vec{a}\forkindep[C]B$. Then by Lemma \ref{CommtingLemmaForking} for all $1 \leq j\leq n$ and $1\leq i \leq k$ we have 
\[
\text{Pr}_{\overline{B\cup C}}P_ia_j = \text{Pr}_{\overline{C}}P_ia_j \text{ if and only if } P_i\text{Pr}_{\overline{B\cup C}}a_j = P_i\text{Pr}_{\overline{C}}a_j.
\]
Also, for each $v \in H$ we have that $v = \sum_{i=1}^kP_iv$. Then, for all $1 \leq j \leq n$, we have the following equivalence
\[
P_i\text{Pr}_{\overline{B\cup C}}a_j = P_i\text{Pr}_{\overline{C}}a_j \text{ for each } 1 \leq i \leq k  \text{ if and only if } \text{Pr}_{\overline{B\cup C}}a_j = \text{Pr}_{\overline{C}}a_j.
\]
\end{proof}

\begin{prop}\label{forkingTuplesAndProjections}

(Triviality)

Let $\vec{a} = (a_1,\dots, a_n) \in H^n$ and $\vec{b} = (b_1,\dots, b_\ell) \in H^\ell$, and let $C \subset H$ be small. Then, $\vec{a}\forkindep[C]\vec{b}$ if and only if for all  $1\leq i \leq k$, $1\leq j_1 \leq n$ and $1\leq j_2 \leq \ell$ we have $P_i(a_{j_1})\forkindep[C]P_i(b_{j_2})$. 

\end{prop}

\begin{proof}

Assume that $\vec a\forkindep[C]\vec b$, then for all $1\leq i \leq k$ and all $1\leq j_1 \leq n$ we have $P_i(\text{Pr}_{\overline{\Parl{b_1,\dots,b_\ell}\cup C}}(a_{j_1}))= P_i(\text{Pr}_{\overline{C}}(a_{j_1}))$. Since
\[
\overline{C} \subseteq \overline{C \cup \Parl{P_ib_{j_2}}} \subseteq \overline{C\cup \Parl{b_{j_2}}} \subseteq \overline{C\cup \Parl{b_1,\dots, b_\ell}},
\]
we have $P_i(\text{Pr}_{\overline{C\cup\Parl{P_ib_{j_2}}}}(a_{j_1})) = P_i(\text{Pr}_{\overline{C}}(a_{j_1}))$. By Lemma \ref{CommtingLemmaForking}, we have
$P_i(a_{j_1})\forkindep[C]P_m(b_{j_2})$.\\

Now, assume for all $1\leq i \leq k, 1\leq j_1 \leq n$ and $1\leq j_2 \leq \ell$ we have that $P_i a_{j_1}\forkindep[C]P_i b_{j_2}$. Let us write $P_ia_{j_1} = \text{Pr}_{\overline{C}}(P_ia_{j_1}) + \text{Pr}_{\overline{C}^\perp}(P_ia_{j_1})$. By hypothesis the projection of $P_ia_{j_1}$ over $\text{Pr}_{\overline{C}^\perp}(P_ib_{j_2})$ is equal to $0$. Then $\text{Pr}_{\overline{C}^\perp}(P_ia_{j_1})$ is orthogonal to $P_ib_{j_2}$. Similarly, for any $g\in G$ we will obtain that $\text{Pr}_{\overline{C}^\perp}(P_ia_{j_1})$ is orthogonal to $\pi(g)(P_ib_{j_2})$. This implies that $\text{Pr}_{\overline{C}^\perp}(P_ia_{j_1})\!\perp\!\overline{C \cup \Parl{P_ib_{j_2}}}$.  The above orthogonality relation holds for each $b_{j_2}$ with $1 \leq j_2 \leq \ell$, making the projection of $\text{Pr}_{\overline{C}^\perp}(P_ia_{j_1})$ on $\overline{C \cup \Parl{P_ib_1, \dots, P_ib_\ell}}$ equal to 0. Then we obtain 
\[
\text{Pr}_{\overline{C}}(P_ia_{j_1}) = \text{Pr}_{\overline{C\cup\Parl{P_ib_1, \dots, P_ib_\ell}}}(P_ia_{j_1}). 
\]
Observe that the $i$-th component $H_i$ of $H$ is closed under the action of the operators $\Parl{\pi(g)}_{g\in G}$ and subspace projections. Then
\[
\text{Pr}_{\overline{C\cup\Parl{P_ib_1, \dots, P_ib_\ell}}}(P_ia_{j_1}) = \text{Pr}_{\overline{C\cup\Parl{b_1, \dots, b_\ell}}}(P_ia_{j_1}).
\]
Thus, $\text{Pr}_{\overline{C}}(P_ia_{j_1}) = \text{Pr}_{\overline{C\cup\Parl{b_1, \dots, b_\ell}}}(P_ia_{j_1})$ for all $1\leq i \leq k$ and $1\leq j_1 \leq n$.

\end{proof}

\begin{theo}\label{thm:charindepence} Let $(K, \pi)\models IHS_\pi$ be $\kappa$-saturated. Then the notion $*$-independence agrees with non-forking and non-forking is trivial.   
\end{theo}

\begin{proof}
It is enough to show that $*$-independence satisfies finite character, local character, transitivity, symmetry, invariance, existence, and stationarity. We check finite character, the other properties can be easily checked using the approach from \cite{Argoty}.

Finite character: let $\vec a=(a_1,\dots,a_n)\in K^n$ be a finite tuple, and let $B, C\subset K$ be small. We prove that if $\vec a\forkindep_C B_0$ for all finite $B_0\subseteq B$ then $\vec a \forkindep_C B$. Note that if $\vec{a}\forkindep_C B_0$, then $\Pr_{\overline{B_0\cup C}}(a_j) = \Pr_{\overline{C}} (a_j)$ for all $1\leq j \leq n$. If this happens for all finite $B_0\subseteq B$ then $\Pr_{\overline{B\cup C}}(a_j) = \Pr_{\overline{C}} (a_j)$ for all $1\leq j\leq n$ as desired.

Finally, triviality of forking follows from the previous result and Proposition \ref{forkingTuplesAndProjections}.
\end{proof}

%Symmetry: Follows from Lemma \ref{almostsymmetry}.

%Transitivity: It follows from Corollary \ref{forkingEquivalenceGFinite}.

%Stationarity. We show stationarity of types over algebraically closed subsets of $K$. So let $\overline{a}=(a_1,\dots,a_n)\in K^n$, $\overline{d}=(d_1,\dots,d_n)\in K^n$ , let $B,C\subseteq K$ be small sets.
%Assume that $B=\overline{B}$, $\tp(\overline{a}/B)=\tp(\overline{a}/B)$
%and that $\overline{a} \forkindep_B C$, $\overline{d} \forkindep_B C$. We prove that $\tp(\overline{a}/B\cup C)=\tp(\overline{a}/B\cup C)$

%From the definition of $*$-independence and Observation \ref{projectagree} we have $$\Pr\! _{B\cup C} (P_m(a_i)) = \Pr\! _B (P_m(a_i)) = \Pr\! _B (P_m(d_i))= \Pr\! _{B\cup C} (P_m(d_i))$$

%From this fact it is easy to conclude that $\tp(\overline{a}/B\cup C)=\tp(\overline{a}/B\cup C)$
%\end{proof}

There are some easy applications of our characterization of non-forking, among them:

\begin{prop}
    The theory $IHS_\pi$
    is non-multidimensional.
\end{prop}

\begin{proof}
    It suffices to prove that any non-algebraic stationary type is not orthogonal to a type over $\emptyset$. Let $a\in H$, let $C\subseteq H$ be small and algebraically closed. Consider $p=\tp(a/C)$ and $q=\tp(a-\text{Pr}_{C}(a)/\emptyset)$.
    Then, if $p$ is non-algebraic, we have $a-\text{Pr}_{C}(a)\neq 0$ and clearly $a-\text{Pr}_{C}(a)\not \forkindep[C] a$.
\end{proof}

We can classify the models of $IHS_\pi$ in terms of the density character of the irreducible representations that appear in the model. This gives a classification of models of $IHS_\pi$ in terms of finitely many cardinals. We will now study  more ``geometric complexity'' aspects of the theory. For this we need:

\begin{defn}
Let $(H,\pi)$ be a separable model of $IHS_\pi$, and let $K\subseteq H$ be a closed subspace, which is invariant under the action of $\pi$ and is such that $(H,\pi)\succcurlyeq(K,\pi\!\!\restriction_{K})$. Let $P$ be the predicate on $H$ that measures the distance to the subspace $K$. Then, if $(H,\pi)$ is $\Ao$-saturated over $(K,\pi\!\!\restriction_{K})$ and $(K,\pi\!\!\restriction_{K})$ is $\Ao$-saturated, we call the pair of structures $((H,\pi),P)$ in the language $\mathcal{L}_\pi\cup\parl{P}$ a \textit{belle paire}. We write $T_{\pi P}$ for the theory of 
belles paires of models of $IHS_\pi$. Sometimes instead of writing $((H,\pi),P)$ we will abuse notation and write $((H,\pi),K)$ for the same structure.  
\end{defn}
Belle paires were first defined in first order by Poizat in \cite{PoiBelle}. There are many applications of belles paires, among them the work of the first named author of the paper with Ben Yaacov and Henson around the notion of the topology of convergence of canonical bases \cite{BYBH}.
Recall from \cite{BYBH} that a stable theory is \emph{SFB (strongly finitely based)} if the topology of convergence of canonical bases coincides with the distance topology on the space of types over models. This notion is a reasonable continuous analogue to the notion of $1$-basedness for stable first order theories, for more details see \cite{BYBH}. 
In this paper we will need belle paires for the following result:

\begin{fact}[Theorem 3.10 \cite{BYBH}]\label{fact:SBFEquivalence}
Let $T$ be any stable continuous theory. Then $T_P$ is $\Ao$-categorical if and only if
$T$ is $\Ao$-categorical and SFB.    
\end{fact}

Our next goal is to show that $IHS_\pi$ has SFB. We already know $IHS_\pi$ is $\Ao$-categorical and stable, so by Fact \ref{fact:SBFEquivalence}, it remains to prove that $T_{\pi P}$ is $\Ao$-categorical. In order to show this, let $(H,\pi)$, $(K,\rho)$ be separable models of $IHS_\pi$, by Corollary \ref{cor:wCategoricityGFinite}, we have a complete description of these models in terms of the invariant subspaces, so  we can write
\[
H = H^{\text{inf}}\oplus H^{\text{fin}} \text{ where } H^{\text{fin}} =  V_{m+1}\oplus \ldots \oplus V_{k},
\]
and
\[
K = K^{\text{inf}}\oplus K^{\text{fin}} \text{ where } K^{\text{fin}} =  V_{m+1}\oplus \ldots \oplus V_{k}.
\]
Notice that the finite dimensional components of the models are  isomorphic as representations of $G$ in finite dimensional Hilbert spaces.
\begin{lem}\label{BeautifulFIniteLemma}

Let $(H,\pi)\succcurlyeq(K,\rho)$ and assume $((H,\pi),K)$ is a separable \textit{belle paire} of models of $IHS_\pi$. Write 
$K^\perp$ for the subspace $K^\perp \cap H$ of $H$, which is invariant under the action of $\pi$. Then, 
\[
(K^\perp,\pi\!\!\restriction_{K^\perp}) \cong (H^{\text{inf}}, \pi\!\!\restriction_{H^{\text{inf}}}).
\]

\end{lem}

\begin{proof}
Recall that the predicate $P(v) = \min_{w \in K} \norm{v-w}$ measures the distance to $K$. Then $\Pr_{K} (v) = \text{arg min }P(v)$ (the projection of $v$ in the subspace $K$) is definible in the extended language $\mathcal{L}_\pi\cup\Parl{P}$ (see for example Proposition 2.4 \cite{BHV} for a proof). It is also easy to see that the distance from a vector $v$ to $K^{\perp}$ is given by $\sqrt{\|v\|^2-P(v)^2}$ and so we can quantify over $K^{\perp}$ and we get that the projection over $K^{\perp}$ is definable as well.  The space $Pr_{K}(H)^\perp = K^\perp$ only has copies of the irreducible representations of $G$ appearing in $H^{\text{inf}}$. Moreover, since $(H,\pi)\succcurlyeq(K,\rho)$ is an $\Ao$-saturated extension, each irreducible representation appearing in $H^{\text{inf}}$ also appears in $K^\perp$ and has dimension equal to $\Ao$. From this we get the desired isomorphism.

\begin{rem}\label{rm:infdim}
Notice that since we can quantify over $K^\perp$, we can express that $\dim (H_i)\cap K^\perp\geq \Ao$ for $1\leq i \leq m$ as a scheme of sentences that belong $T_{\pi P}$. 
\end{rem}

\end{proof}

\begin{thm}\label{BellePaires-cat}
The theory of \textit{belle paires} of $IHS_\pi$ is $\Ao$-categorical.
\end{thm}

\begin{proof}

Let $((K_1,\rho),K_2)$ and $((H_1,\pi),H_2)$ be two separable models of the theory $T_{\pi P}$. Then, by Lemma \ref{BeautifulFIniteLemma} and Remark \ref{rm:infdim} we have that $(K^\perp_2, \rho\!\!\restriction_{K^\perp_2})$ and $(H^\perp_2,\pi\!\!\restriction_{H^\perp_2})$ are isomorphic as representations. Additionally, the models $(K_2,\rho\!\!\restriction_{K_2})$
and $(H_2,\pi\!\!\restriction_{H_2})$ are isomorphic as they are separable models of $IHS_\pi$. Then, the expansions $((K_1,\rho),K_2)$ and $((H_1,\pi),H_2)$ are also isomorphic.
\end{proof}

\begin{cor}
The theory $IHS_\pi$ has SFB.
\end{cor}

\begin{proof}
  The result follows from Fact \ref{fact:SBFEquivalence} and 
  Theorem \ref{BellePaires-cat}.
\end{proof}

 One can change perspective and follow the ideas from \cite{ItaiIsaac} and consider actions by compact groups instead of finite groups and generalize results of $G$-actions to that setting. A natural starting point would be:

\begin{ques} Assume $G$ is a compact group.
     Can one characterize again the existentially closed expansions in terms of the left regular representations? Do irreducible representations play the same role in this setting as they did for finite groups?
\end{ques}

\section{Hilbert spaces expanded by a representation of infinite groups}\label{rep:infinite}

In this section $G$ will denote a discrete infinite countable group and we will fix $\parl{g_n}_{n \in \Na}$ an enumeration of $G$. Additionally, $H$ will be a infinite dimensional Hilbert space, and $\pi : G \longrightarrow U(H)$ will denote a unitary representation of $G$. In this setting, we first give some examples where the theory $IHS_\pi$ is either $\Ao$-categorical or only $\Ao$-categorical up to perturbations. Then, we prove the general result for $IHS_\pi$, which states that regardless of the nature of $G$ or $\pi$, the theory $IHS_\pi$ is $\Ao$-categorical up to perturbations. Finally, we prove that when we also assume that $IHS_\pi$ is model complete, then $IHS_\pi$ is $\Ao$-stable up to perturbations. 

\begin{example}\label{finitecostumoofinfinite} Suppose that $\pi : G \longrightarrow U(H)$ has finite image. Then, the isomorphism $\quotient{G}{\ker(\pi)} \cong$ Im$(\pi)$, implies that the unitary irreducible representations of $\pi$ are in correspondence with the irreducible representations of the group $\quotient{G}{\ker(\pi)}$. In this case, we can apply the results from the previous section and by Theorem \ref{GFinitioAocategoricidad} the theory $IHS_\pi$ is $\Ao$-categorical. 

\end{example}

On the other hand, having nonempty continuous spectrum (see  Definition \ref{Defn:spectrum} and the corresponding notation) in one of the operators belonging to the representation of $G$ over $H$ allows us to construct two separable non-isomorphic models.
\begin{prop}

Let $(H_1,\pi_1)$ be a separable model of $IHS_\pi$. Suppose that there is $g \in G$ such that $\sigma(\pi_1(g))\setminus\sigma_{\text{p}}(\pi_1(g))\not=\emptyset$. Then the theory $IHS_\pi$ is not $\Ao$-categorical.

\end{prop}

\begin{proof}
Let $g \in G$ be as in the hypothesis. By Fact \ref{Unitarioespectroresidual} there is $\lambda \in \sigma_c(\pi(g))$ and thus we can find a sequence $\Parl{v_n}_{n \in \Na} \subseteq H$ of normal vectors such that 
\[
\Lim{n}{\infty} \Norm{\pi_1(g)v_n - \lambda v_n}_2 = 0.
\]
Let $\mathcal{F}$ be a non-principal ultrafilter over $\mathbb{N}$ and define $\mathcal{M} := \Pi_{n,\mathcal{F}}(H_1,\pi_1)$. Then, the element $[(v_n)_n]$ is normal and satisfies 
\[
\pi^{\mathcal{M}}(g)[(v_n)_n] = [(\pi_1(g)v_n)_n] = \lambda [(v_n)_n].
\]
Thus, $\lambda\in \sigma_p(\pi^{\mathcal{M}}(g))$. By L\"owenheim-Skolem there exists a separable model $(H_2,\pi_2)$ of $IHS_\pi$ where $\lambda$ is in the punctual spectrum of $\pi_2(g)$. Then the representations $(H_1,\pi_1)$ and $(H_2,\pi_2)$ of $G$ are not isomorphic, so the theory is not $\Ao$-categorical.
\end{proof}

Modulo perturbations, we get a simpler picture that does not depend on the spectrum of the operators $\pi(g_n)$. We start with a technical lemma:

\begin{lem}\label{lem:definabilityofoperators}

Let $(H,\pi)$ be a model of $IHS_\pi$, and for each $g_n \in G$ define $U_n := \pi(g_n)$. Also, let $\mathcal{A}$ be the C$^*$-algebra generated by $\parl{U_n}_{n \in \Na}$. Then all the operators in $\mathcal{A}$ are definable in the language $\mathcal{L}_\pi$.

\end{lem}

\begin{proof}

We denote by $\mathcal{A}_0$ the $*$-algebra generated by $\parl{U_n}_{n \in \Na}$. Since the product of $U_k$ with $U_m$ is in $\parl{U_n}_{n \in \Na}$, any element $T \in \mathcal{A}_0$ can be expressed as $T = \sum_{i=1}^n \lambda_iU_{i}$ and thus it is definable in $\mathcal{L}_{\pi}$. Observe that if we take the topological closure of $\mathcal{A}_0$ in the operator topology, we obtain $\mathcal{A}$. Since the topology in $B(H)$ is the generated by the norm:
\[
\norm{T} = \sup_{\norm{x}\leq 1} \|Tx\| \text{ where } T \in B(H),
\]
if $T \in \mathcal{A}$, then $T$ is the limit of sums of the form $T_m = \sum_{i=1}^m \lambda_iU_{i}$. Thus, we can write $T = \sum_{i=1}^\infty \lambda_iU_{i}$ which satisfy for every $x, y$ in the unit ball of $H$ the following
\[
|\|T_mx - y\| -  \|Tx - y\||\leq  \|Tx - T_mx\| \leq \|T - T_m\|.
\]
Hence the sequence $\parl{\norm{T_mx - y}}_{m \in \Na}$ converges uniformly to $\norm{Tx - y}$ in the unit ball of $H$. Thus, the function $T : H \longrightarrow H$ is definable.\\ 

\end{proof}

For the next results, recall Definition 
\ref{defn:pertubations} and Remark \ref{rem:satimpliesstab}.

\begin{thm}\label{AocategoriciadsalvopGtimesG}

The theory $IHS_\pi$ is $\Ao$-categorical up to perturbations.

\end{thm}

\begin{proof}

Let $(H_1,\pi_1)$ and $(H_2,\pi_2)$ be separable models of $IHS_\pi$. For each $n \in \Na$ define $U_n := \pi_1(g_n)$ and $V_n := \pi_2(g_n)$, also we denote by $\phi$ the $*$-morphism induced by the assignment $\phi(U_n) = V_n$. Let $\mathcal{A}$ and $\mathcal{B}$ be the C$^*$-algebras generated by $\parl{U_n}_{n \in \Na}$ and $\parl{V_n}_{n \in \Na}$ respectively, then the extension $\Phi$ of $\phi$ to $\mathcal{A}$ is a representation of $\mathcal{A}$ in $B(H_2)$ with image equal to $\mathcal{B}$. We are dealing with two representations of $\mathcal{A}$, first $id_{\mathcal{A}}$ the representation that sends each $T\in \mathcal{A}$ to itself, and $\Phi$, the representation induced by $\phi$.  If we prove that for all $T \in \mathcal{A}$ we have that rank$(T) = $ rank$(\Phi(T))$, then by Fact \ref{fact:unitarilyapproximatelyequivalentwithcompacts} the representaions $id_{\mathcal{A}}$ and $\Phi$ are approximately unitarily equivalent. 

Recall that in Lemma \ref{lem:definabilityofoperators} we proved that any $T \in \mathcal{A}$ is definable in the language $\mathcal{L}_\pi$. We will now prove that the rank of the operator $T$ is coded in the theory $IHS_\pi$.\\

\hspace{-0,4cm}\textbf{Case 1:} Suppose rank$(T) = m$. Then, the sentence
\begin{equation}\label{eq:finiterank}
\inf_{v_1\dots v_m}\sup_{v}\max\{\max_{i\leq m}|\|Tv_i\|-1|,\max_{i<j\leq m}|\gen{Tv_i,Tv_j}|\},\|Tv-(\text{Pr}_{Tv_1}v+\ldots +\text{Pr}_{Tv_m}v)\|\} = 0
\end{equation}
is part of the theory $IHS_\pi$, and it axiomatizes rank$(T)=m$.\\

\hspace{-0,4cm}\textbf{Case 2:} Suppose rank$(T)=\Ao$. Consider the following scheme indexed by $m\in \Na^{>0}$
\begin{equation}\label{eq:infiniterank}
\inf_{v_1\dots v_m}\max\{\max_{i\leq m }|\|Tv_i\|-1| , \max_{i<j\leq m}|\gen{Tv_i,Tv_j}|\} = 0
\end{equation}
Then, the scheme \ref{eq:infiniterank} is part of the theory $IHS_\pi$, and axiomatizes rank$(T)=\infty$.\\

\hspace{-0.4cm}The sentence \ref{eq:infiniterank}
and the scheme \ref{eq:finiterank} imply that if $T \in \mathcal{A}$ and $m \in \parl{0,1, \ldots, \Ao}$ is such that rank$(T) = m$,  then rank$(\Phi(T)) = m$, because $\Phi(T)$ is the interpretation of $T$ in $(H_2,\pi_2)$ which models $IHS_\pi$. Hence $id_{\mathcal{A}}$ and $\Phi$ are AUE, implying that $(H_1,\pi_1)$ and $(H_2,\pi_2)$ are approximately isomorphic, thus the theory $IHS_\pi$ is $\Ao$-categorical up to perturbations.

\end{proof}

\begin{thm}\label{Ao-stabilitypertubations}
Assume the theory $IHS_\pi$ is model-complete. Then the theory $IHS_\pi$ is $\Ao$-stable up to perturbations.

\end{thm}

\begin{proof}
Let $(H,\pi)$ be a separable model of $IHS_\pi$ and let $\mathcal{A}$ be the C$^*$-algebra generated by $\parl{\pi(g_n)}_{n\in \Na}$. By Lemma \ref{lem:compactnoncompact} we can write \[
(H,\mathcal{A})\cong(H_c\oplus H^\perp_c, \mathcal{A}_c\oplus\mathcal{A}_c^\perp), 
\]
where the subalgebra $\mathcal{A}_c$ are composed only by compact operators,
and the subalgebra $\mathcal{A}_c^\perp$ has no compact operators. Also, $\mathcal{A}_c$ and $\mathcal{A}_c^\perp$ act over $H_c$ and $H^\perp_c$ respectively. The algebra $\mathcal{A}_c$ is the topological closure of the $*\text{-algebra}$ generated by the family
\[
\{\text{Pr}_{H_c}\pi(g_n)\!\!\restriction\!\!_{H_c}\}_{n \in \Na}
\]
In the same way, the algebra $\mathcal{A}^\perp_c$ is the topological closure of the $*\text{-algebra}$ generated by the family 
\[
\{ \text{Pr}_{H^\perp_c}\pi(g_n)\!\!\restriction\!\!_{H^\perp_c}\}_{n \in \Na}.
\]
Now, if we write $\pi_c^\perp$ for the restriction of $\pi$ to $H^\perp_c$, we can define the representation $(H_1,\tau)=\bigoplus_{i\in \omega} (H^\perp_c, \pi_c^\perp)$. This representation is a Hilbert space with an action where all operators have rank $\aleph_0$. Finally let $(H\oplus H_1, \pi\oplus \tau)$ be the representation coming from the direct sum.\\

\hspace{-0,5cm} \textbf{Claim  1:} $(H\oplus H_1, \pi\oplus \tau)\models IHS_\pi$.

\hspace{-0,5cm} Let $\mathcal{B}$ the C$^*$-algebra generated by the operators $\parl{\pi\oplus\tau(g_n)}_{n\in \Na}$. We denote by $\phi$ the $*$-morphism induced by the assignment $\phi(\pi(g_n)) = \pi\oplus \tau(g_n)$. The extension $\Phi$ of $\phi$ to $\mathcal{A}$ obtained by linearity and continuity is a representation of $\mathcal{A}$ in $B(H\oplus H_1)$ whose image is $\mathcal{B}$. In this setting we have again two representations of $\mathcal{A}$, first $id_{\mathcal{A}}$ the representation that sends each $T\in \mathcal{A}$ to itself, and $\Phi$.  The subalgebra of compact operators of both algebras $\mathcal{A}$ and $\mathcal{B}$, appear in the copy of $(H_c,\pi_c)$ inside each sum, and by Fact \ref{fact:unitarilyequivalentcompact} they are isomorphic; the non-compact operators, which appear in $(H^\perp_c, \pi_c^\perp)$ and in $(H_1,\tau)$ all have rank $\aleph_0$. It follows by Fact \ref{fact:unitarilyapproximatelyequivalentwithcompacts} that the representations $(H,\pi)$ and $(H\oplus H_1,\pi\oplus \tau)$ are AUE and thus satisfy the same theory $IHS_\pi$.

Since $IHS_\pi$ is model complete, we have 
\[
(H, \pi)\preccurlyeq (H\oplus H_1, \pi\oplus \tau).
\]
We will prove that the elementary superstructure 
$(H\oplus H_1, \pi\oplus \tau)$ is $\aleph_0$-saturated up to perturbations over $(H,\pi)$ and thus, since $(H,\pi)$ was any separable model of $IHS_\pi$, this shows that $IHS_\pi$ is $\aleph_0$-stable up to perturbations.\

\hspace{-0.4cm}Let $(K,\rho) \succcurlyeq (H, \pi)$ be an elementary separable superstructure. As in Claim 1, we also have that $(K\oplus H_1, \rho\oplus \tau)\models IHS_\pi$ and since $IHS_\pi$ is model-complete, $(K, \rho)\preccurlyeq(K\oplus H_1, \rho\oplus \tau)$. By construction of $(K\oplus H_1, \rho\oplus \tau)$, we can write $K\oplus H_1 \cong H \oplus H^\perp$ and $\rho\oplus \tau \cong \pi \oplus \rho'$, where $\rho' = \rho\!\!\upharpoonright_{H^\perp}$. By construction, the C$^*$-algebras induced by the representations $\tau$ and $\rho'$ over $H_1$ and $H^\perp$ respectively, are free of compact operators.
We get again using Fact \ref{fact:unitarilyapproximatelyequivalentwithcompacts} that these two representations are approximately unitarily equivalent and so for every $\varepsilon > 0$ there is a unitary map $\mathcal{O}_\varepsilon:H_1 \longrightarrow H^\perp$ such that 
$\lim_{\varepsilon\to 0}
\Norm{\pi_2(g) - \mathcal{O}_\varepsilon\pi_1(g)\mathcal{O}^*_\varepsilon } =0$ for each $g\in G$. Then for all $\varepsilon > 0$ we have the following diagram

\begin{center}
\begin{tikzcd}
  (H\oplus H_1,\pi\oplus \tau)  & (H\oplus H^\perp,\pi\oplus \rho') \ar[l, dashed, swap, "id\oplus \mathcal{O}^*_\varepsilon"] \\
  (H,\pi) \ar{u}{\rotatebox[origin=c]{90}{$\preceq$}} \ar{r}[swap]{\preceq} & (K,\rho)  \ar{u}[swap]{\rotatebox[origin=c]{90}{$\preceq$}}
\end{tikzcd}
\end{center}
where the map $id$ is the identity map over $H$. 

%Recall that $b\models p(x)$ with $b\in K$ and define $b_\varepsilon=id\oplus \mathcal{O}^*_\varepsilon(b)$. Then the sequence $\Parl{b_{\frac{1}{n}}}_{n\in\Na}$ belongs to the space $H\oplus H_1$ and it is a witness that $(H\oplus H_1,\pi\oplus\tau)$ realizes, up to perturbations, the type $p(x)$. 
\end{proof}

\begin{obse}
    Let $G$ be a finite group, and let $\pi$ be a representation of $G$ on an infinite dimensional Hilbert space $H$. Following the notation of section \ref{rep:finite} we can write $H=H^{\text{fin}}\oplus H^{\text{inf}}$. Using the notation from Theorem \ref{Ao-stabilitypertubations}, we have $H_c=H^{\text{fin}}$ and $H_c^\perp=H^{\text{inf}}$.
\end{obse}

\begin{example}
Let $G$ be a countable amenable group, and consider
$Th(\infty \ell_2(G),\infty \lambda_G)$. By Theorem 2.8 in \cite{AlexHGGA} this theory is the model companion of the theory of $G$-representations and thus it is model complete. By Theorem \ref{Ao-stabilitypertubations} we get that $Th(\infty \ell_2(G),\infty \lambda_G)$ is $\Ao$-stable up to perturbations; it was already known by \cite[Section 3]{AlexHGGA} that it is superstable.

The especial case where $G=\mathbb{Z}$ was
considered in \cite{zadka} and the model companion was characterized as the collection of expansions $(H,\tau(n):n\in \mathbb{Z})$ where the spectrum of $\tau(1)$ is $S^1$. In \cite{itai} it was proved that this expansion is $\aleph_0$-stable up to perturbations, a especial case
of Theorem \ref{Ao-stabilitypertubations}.
\end{example}

%%%%%%%%%%%%%%%%%%%%%%%%%%%%%%%%%%%%%%%%%%%%%%%%%%%%%%%%%%%%%%%%%%%%%%%%%%%%%%%%%%%%%%%%

\end{document}